\documentclass[a4paper,11pt]{article}
\usepackage{amsmath, amsfonts, amscd, amssymb, amsthm, enumerate, xypic}

\def\bfB{\mathbf{B}}

\def\ortho{o}

\DeclareMathOperator{\Fbar}{\overline{\mathbb{F}}}
\DeclareMathOperator{\id}{\operatorname{id}}

\DeclareMathOperator{\card}{\#\,}
\DeclareMathOperator{\Mat}{\operatorname{M}}
\DeclareMathOperator{\Mats}{\operatorname{S}}
\DeclareMathOperator{\Mata}{\operatorname{A}}
\DeclareMathOperator{\End}{\operatorname{End}}

\DeclareMathOperator{\NT}{\operatorname{NT}}
\DeclareMathOperator{\GL}{\operatorname{GL}}
\DeclareMathOperator{\Ker}{\operatorname{Ker}}

\DeclareMathOperator{\trk}{\operatorname{trk}}
\DeclareMathOperator{\Hom}{\operatorname{Hom}}
\DeclareMathOperator{\Vect}{\operatorname{span}}
\DeclareMathOperator{\im}{\operatorname{Im}}

\DeclareMathOperator{\tr}{\operatorname{tr}}
\DeclareMathOperator{\car}{\operatorname{char}}
\DeclareMathOperator{\Sp}{\operatorname{Sp}}
\DeclareMathOperator{\rk}{\operatorname{rk}}

\renewcommand{\setminus}{\smallsetminus}
\renewcommand{\epsilon}{\varepsilon}

\def\F{\mathbb{F}}
\def\K{\mathbb{K}}

\def\Q{\mathbb{Q}}

\def\Z{\mathbb{Z}}

\renewcommand{\L}{\mathbb{L}}

\def\calA{\mathcal{A}}
\def\calB{\mathcal{B}}
\def\calC{\mathcal{C}}

\def\calM{\mathcal{M}}

\def\calS{\mathcal{S}}
\def\calT{\mathcal{T}}
\def\calU{\mathcal{U}}
\def\calV{\mathcal{V}}
\def\calW{\mathcal{W}}
\def\calX{\mathcal{X}}

\def\lcro{\mathopen{[\![}}
\def\rcro{\mathclose{]\!]}}

\theoremstyle{definition}
\newtheorem{Def}{Definition}[section]
\newtheorem{Not}[Def]{Notation}

\theoremstyle{plain}
\newtheorem{theo}{Theorem}[section]
\newtheorem{prop}[theo]{Proposition}

\newtheorem{lemma}[theo]{Lemma}
\newtheorem{claim}{Claim}[section]

\theoremstyle{plain}
\newtheorem{conj}{Conjecture}[section]

\theoremstyle{remark}
\newtheorem{Rems}{Remarks}[section]
\newtheorem{Rem}[Rems]{Remark}

\title{Spaces of matrices with few eigenvalues (II)}
\author{Cl\'ement de Seguins Pazzis\footnote{Universit\'e de Versailles Saint-Quentin-en-Yvelines, Laboratoire de Math\'ematiques
de Versailles, 45 avenue des Etats-Unis, 78035 Versailles cedex, France}
\footnote{e-mail address: clement.de-seguins-pazzis@ac-versailles.fr}}

\begin{document}

\thispagestyle{plain}

\maketitle
\begin{abstract}
Let $\F$ be a field, and $\calM$ be a linear subspace of matrices of $\Mat_n(\F)$ that have at most two eigenvalues in $\F$
(respectively, at most one non-zero eigenvalue in $\F$). In a previous article, we have determined the greatest possible dimension for $\calM$ when the characteristic of $\F$ is not $2$.
In this article, we solve this problem for all fields with characteristic $2$.
\end{abstract}



% Relectures achevées. Prêt à être soumis à LAA

\section{Introduction and main results}

\subsection{Main issues}

In this article, we let $\F$ be an arbitrary field and we choose an algebraic closure $\Fbar$ of it.
We denote by $\Mat_n(\F)$ the algebra of all $n$-by-$n$ square matrices with entries in $\F$, by
$\GL_n(\F)$ the group of invertible elements of $\Mat_n(\F)$,
by $\NT_n(\F)$ the linear subspace of $\Mat_n(\F)$ consisting of its strictly upper-triangular matrices,
by $\Mats_n(\F)$ the linear subspace consisting of its symmetric matrices, by $\Mata_n(\F)$ the linear subspace consisting
of its alternating matrices (i.e., skew-symmetric matrices with zero entries on the diagonal)
and by $\mathfrak{sl}_n(\F)$ the linear subspace of $\Mat_n(\F)$ consisting of its trace zero matrices.
Given a finite-dimensional $\F$-algebra $\calA$, an element $x$ of $\calA$ and an extension $\L$ of the field $\F$,
the $\L$-spectrum of $x$, denoted by $\Sp_\L(x)$, is the set of all roots in $\L$ of the minimal polynomial of $x$.

Two subsets $\calV$ and $\calW$ of $\Mat_n(\F)$ are called \textbf{similar},
and we write $\calV \simeq \calW$, if $\calW=P\calV P^{-1}$ for some $P \in \GL_n(\F)$
(this means that $\calV$ and $\calW$ represent, in different bases, the same set of endomorphisms of $\F^n$).

The present article deals with extensions of Gerstenhaber's seminal work on
spaces of nilpotent matrices \cite{Gerstenhaber}. Following Radwan and Loewy,
we extend the scope to spaces of matrices with an upper bound on the cardinality of the spectrum.
Let $\calS$ be a linear subspace of an $\F$-algebra $\calA$.
We say that $\calS$ is :
\begin{itemize}
\item a \textbf{$k$-spec space} when $\card \Sp_\F(x) \leq k$ for all $x \in \calS$;
\item a \textbf{$\overline{k}$-spec space} when $\card \Sp_{\Fbar}(x) \leq k$ for all $x \in \calS$;
\item a \textbf{$k^\star$-spec space} when $\card (\Sp_\F(x) \setminus \{0\}) \leq k$ for all $x \in \calS$;
\item a \textbf{$\overline{k}^\star$-spec space} when $\card (\Sp_{\Fbar}(x) \setminus \{0\}) \leq k$ for all $x \in \calS$.
\end{itemize}
A $k$-spec linear subspace (respectively, a $\overline{k}$-spec space, etc) is called \textbf{optimal} when it has the greatest possible
dimension among the $k$-spec linear subspaces (respectively, the $\overline{k}$-spec linear subspaces, etc) of the algebra under consideration.

In particular, the $\overline{0}^\star$-spec linear subspaces of endomorphisms (respectively, of matrices) are also the spaces of nilpotent endomorphisms
(respectively, of nilpotent matrices); the $0^\star$-spec linear subspaces are sometimes called the \emph{trivial spectrum subspaces} (see e.g.\ \cite{dSPgivenrank}).
Here is the seminal result on spaces of nilpotent matrices:

\begin{theo}[Gerstenhaber \cite{Gerstenhaber}, Serezhkin \cite{Serezhkin}]
The maximal possible dimension for a $\overline{0}^\star$-spec subspace of $\Mat_n(\F)$ is $\dbinom{n}{2}$,
and the optimal $\overline{0}^\star$-spec subspaces of $\Mat_n(\F)$ are the conjugates of $\NT_n(\F)$.
\end{theo}

In geometric terms, this means that for an $n$-dimensional vector space $V$, every $\overline{0}^\star$-spec subspace $\calS$ of
$\End(V)$ has its dimension less than or equal to $\dbinom{n}{2}$, and equality occurs if and only if there is a complete flag
$\{0\} \subset V_0 \subset \cdots \subset V_n=V$ of linear subspaces such that $\calS$ is the set of all $u \in \End(V)$ that map
$V_i$ into $V_{i-1}$ for all $i \in \lcro 1,n\rcro$.

Here is a partial extension to $0^\star$-spec spaces:

\begin{theo}[de Seguins Pazzis \cite{dSPfeweigenvalues}, Quinlan \cite{Quinlan}]\label{theo:0starspec}
The maximal possible dimension for a $0^\star$-spec subspace of $\Mat_n(\F)$ is $\dbinom{n}{2}$.
\end{theo}

See also \cite{dSPlargeaffinenonsingular} for a full classification of the optimal $0^\star$-spec subspaces of $\Mat_n(\F)$ provided that $|\F|>2$.

The $1$-spec subspaces were studied in \cite{dSPsoleeigenvalue}:

\begin{theo}[de Seguins Pazzis \cite{dSPsoleeigenvalue}]\label{theo:1spec}
The greatest possible dimension for a $1$-spec subspace of $\Mat_n(\F)$ is $\dbinom{n}{2}+1$ unless $n=\car(\F)=2$,
in which case the greatest possible dimension is $3$ and $\mathfrak{sl}_2(\F)$ is the sole optimal $1$-spec subspace of $\Mat_2(\F)$.
\end{theo}

Note that if $\car(\F)=2$ then a $2$-by-$2$ matrix with two distinct eigenvalues in $\Fbar$ must have its trace nonzero.
In the remaining cases, a typical optimal $1$-spec subspace of $\Mat_n(\F)$ is $\F I_n \oplus \NT_n(\F)$,
or more generally $\F I_n\oplus \calT$ where $\calT$ is an optimal $0^\star$-spec subspace.
However, there are examples of optimal $1$-spec subspaces that are not of this type (see \cite{dSPsoleeigenvalue} for various examples).

From the above considerations, it is obvious that exactly the same bounds hold for $\overline{1}$-spec subspaces. The classification of the optimal $\overline{1}$-spec subspaces
was achieved in \cite{dSPsoleeigenvalue} over all fields, with an error for $n=4$ and fields with characteristic $2$ that was recently corrected
\cite{dSPsoleeigenvaluecorr}.

In \cite{dSPfeweigenvalues}, we went even further by studying $1^\star$-spec and $2$-spec spaces over fields with characteristic other than $2$. Here was our main result then:

\begin{theo}[Theorems 1.4 and 1.5 of \cite{dSPfeweigenvalues}]\label{theorem:car<>2}
Assume that $\car(\F) \neq 2$. Let $n \geq 2$. Then:
\begin{enumerate}[(i)]
\item The greatest possible dimension for a $1^\star$-spec subspace of $\Mat_n(\F)$ is $\dbinom{n}{2}+1$, and it is also the
greatest possible dimension for a $\overline{1}^\star$-spec subspace of $\Mat_n(\F)$.
\item If $n>2$ then the greatest possible dimension for a $2$-spec subspace of $\Mat_n(\F)$ is $\dbinom{n}{2}+2$, and it is also the
greatest possible dimension for a $\overline{2}$-spec subspace of $\Mat_n(\F)$.
\end{enumerate}
\end{theo}

There is even a full classification of the optimal $\overline{1}^\star$-spec subspaces of $\Mat_n(\F)$
and of the optimal $\overline{2}$-spec subspaces of $\Mat_n(\F)$, still assuming that $\car(\F) \neq 2$
(see \cite{dSPfeweigenvalues} for details).

Both these results and their proofs heavily rely on the condition $\car(\F) \neq 2$,
and in fact the upper bounds on the dimensions of $1^\star$-spec and $2$-spec spaces no longer hold over fields with characteristic $2$.
Of course, we must immediately rule out fields with $2$ elements, as in that case every linear subspace of $\Mat_n(\F)$ is $1^\star$-spec
and $2$-spec.

At this point, our aims are as expected: for an arbitrary field $\F$ of characteristic $2$, and with more than $2$ elements:
\begin{itemize}
\item Find the greatest possible dimension for a $1^\star$-spec subspace of $\Mat_n(\F)$, and ditto for $2$-spec subspaces.
\item Classify the optimal $\overline{1}^\star$-spec subspaces of $\Mat_n(\F)$, and ditto for the $\overline{2}$-spec subspaces.
\end{itemize}

To state the results, it is useful to introduce the joint of two spaces of square matrices:

\begin{Not}
Let $n$ and $p$ be positive integers, and $\calA$ and $\calC$ be respective subsets of $\Mat_n(\F)$ and $\Mat_p(\F)$.
The \textbf{joint} of $\calA$ and $\calC$, denoted by $\calA \vee \calC$, is the set of all matrices of the form
$$\begin{bmatrix}
A & B \\
0 & C
\end{bmatrix} \quad \text{with $A \in \calA$, $C \in \calC$ and $B \in \Mat_{n,p}(\F)$.}$$
\end{Not}

Noting that this operation is associative, we can also generalize it to whole lists of sets of square matrices.

Assume that $\car(\F)=2$. We have already seen that $\mathfrak{sl}_2(\F)$ is a $\overline{1}$-spec subspace of $\Mat_2(\F)$ with dimension $3$,
and more generally it is clear that, for all $n \geq 2$, the joint
$\mathfrak{sl}_2(\F) \vee \NT_{n-2}(\F)$ is a $\overline{1}^\star$-spec subspace of $\Mat_n(\F)$ with dimension $\dbinom{n}{2}+2$, which is exactly one unit
above the upper bound from Theorem \ref{theorem:car<>2}.

Before we move forward, we must pay tribute to the recent resolution by Omladi\v c and \v Sivic of a conjecture of Loewy and Radwan \cite{LoewyRadwan} on $\overline{k}$-spec subspaces:

\begin{theo}[Omladi\v c and \v Sivic \cite{OmladicSivic}]\label{theorem:OmladicSivic}
Assume that $\F$ is algebraically closed with characteristic $0$.
Let $k \in \lcro 1,n-1\rcro$. Then the greatest possible dimension for a $\overline{k}$-spec subspace of
$\Mat_n(\F)$ is $\dbinom{n}{2}+\dbinom{k}{2}+1$.
\end{theo}

It is even possible to replace the ``algebraically closed field" condition by $\car(\F)=0$ (we will prove this in a future article).
Omladi\v c and \v Sivic also classified the optimal spaces (using the classification for $k=2$ that we had obtained in \cite{dSPfeweigenvalues}).
Alas, the characteristic $0$ provision is unavoidable for this theorem, and the methods also do not allow for a generalization to
$k$-spec spaces, even assuming that $\F$ has characteristic zero. So, although Theorem \ref{theorem:OmladicSivic} is a major achievement, we believe there remains much interesting research to be done on this topic.

\subsection{Main results}

We are now ready to state our main results.
\emph{From that point on, we systematically assume that $\car(\F)=2$ and $|\F|>2$ (and will only repeat this assumption in the statement of the main theorems).}

\begin{theo}\label{theorem:maintheodim1starspec}
Let $n \geq 2$, and $\F$ be a field with $\car(\F)=2$ and $|\F|>2$.
Then the greatest possible dimension for a $1^\star$-spec subspace of $\Mat_n(\F)$ is $\dbinom{n}{2}+2$,
and it is also the greatest possible dimension for a $\overline{1}^\star$-spec subspace.
\end{theo}

As stated earlier, an example of $\overline{1}^\star$-spec subspace of $\Mat_n(\F)$ with dimension $\dbinom{n}{2}+2$ is
$\mathfrak{sl}_2(\F) \vee \NT_{n-2}(\F)$.

For $2$-spec spaces, the situation is further complicated by the presence of a very special situation when $n=4$.
Indeed, the space $\mathfrak{sl}_2(\F) \vee \mathfrak{sl}_2(\F)$ has dimension
$10=\dbinom{4}{2}+4$ and is obviously $2$-spec (and even $\overline{2}$-spec).

\begin{theo}\label{theorem:maintheodim2spec}
Let $n \geq 3$, and $\F$ be a field with $\car(\F)=2$ and $|\F|>2$.
Then the greatest possible dimension for a $2$-spec subspace of $\Mat_n(\F)$ is $\dbinom{n}{2}+3$ if $n\neq 4$,
and $\dbinom{n}{2}+4$ if $n=4$.
The same bound holds if we replace the $2$-spec property with the $\overline{2}$-spec property.
\end{theo}

Before we state the classification of optimal spaces, it is useful that we give another intriguing example of
a $\overline{2}$-spec subspace of $\Mat_4(\F)$ with dimension $10$.
To start with, let us consider a finite-dimensional vector space $V$ with even dimension $2m>0$, equipped with a symplectic form $s$,
i.e., a non-degenerate alternating bilinear form. An endomorphism $u$ of $V$ is called $s$-symmetric
whenever the bilinear form $(x,y) \mapsto s(x,u(y))$ is symmetric i.e., $\forall (x,y)\in V^2, \; s(u(x),y)=-s(x,u(y))$.
The set of all such endomorphisms is denoted by $\calB_s$: it is a linear subspace of $\End(V)$.
Classically, the space $V$ has a symplectic basis $(e_1,\dots,e_m,f_1,\dots,f_m)$, i.e., a basis
such that $s(e_i,f_j)=\delta_{i,j}$ and $s(e_i,e_j)=s(f_i,f_j)=0$ for all $(i,j)\in \lcro 1,m\rcro^2$.
The matrix of $s$ in that basis is
$$K_{2m}:=\begin{bmatrix}
[0]_{m \times m} & I_m \\
-I_m & [0]_{m \times m}
\end{bmatrix}.$$
In such a basis, the elements of $\calB_s$ are represented by the matrices of the form
$(K_{2m})^{-1} S$ with $S \in \Mats_{2m}(\F)$, hence the following block form for these matrices:
$$\begin{bmatrix}
A & S_2 \\
S_1 & -A^T
\end{bmatrix} \quad \text{with $A \in \Mat_m(\F)$ and $(S_1,S_2)\in \Mats_m(\F)^2$.}$$
Now, we recall and reprove a classical result, which holds over all fields:

\begin{prop}
Let $s$ be a symplectic form on a finite-dimensional vector space $V$ over an arbitrary field $\K$.
Let $u \in \End(V)$ be $s$-symmetric. Then the characteristic polynomial $\chi_u$ is an even polynomial.
\end{prop}

\begin{proof}
Let $m \geq 1$.
Let $S \in \Mats_{2m}(\K)$. We need to prove that the characteristic polynomial of $M:=K_{2m}^{-1}S$ is even.
This is obvious if $\car(\F) \neq 2$: in that case indeed we have $M^T=-S K_{2m}^{-1}=-K_{2m} M K_{2m}^{-1}$, and hence
$\chi_{M}(t)=\chi_{M^T}(t)=\chi_{-M}(t)=(-1)^{2m}\chi_M(-t)=\chi_M(-t)$, and the result is proved.

For the general case, we fix an arbitrary integer $k \in \lcro 1,m\rcro$, and denote by $c_k(M)$ the coefficient on $t^{2m-2k+1}$ in the characteristic polynomial of $M$,
which is known to equal the sum of all the principal $(2k-1) \times (2k-1)$ minors of $M$.
Hence there is a polynomial $Q_k \in \Z[t_{i,j}]_{1 \leq i \leq j \leq 2m}$ with integral coefficients
such that, for every field $\K$ and every matrix $S \in \Mats_{2m}(\K)$, one has $c_k(K_{2m}^{-1} S)=Q_k[s_{i,j}]$.
As the first part of the proof applies to $\K=\Q$, we obtain that $Q_k$ vanishes at every specialization of the indeterminates $t_{i,j}$ at rationals, and hence $Q_k=0$. This yields the claimed result.
\end{proof}

Now, let us come back to $\F$ (which has characteristic $2$). Consider a $4$-dimensional vector space $V$ equipped with a symplectic form $s$.
Then $\calB_s$ has dimension $\dbinom{5}{2}=10$. Moreover, for every $u \in \calB_s$, we have $\chi_u(t)=t^4+\alpha t^2+\beta$
for some $(\alpha,\beta)\in \F^2$, and because the Frobenius endomorphism $x \mapsto x^2$ of $\Fbar$ is injective it follows that $u$ has at most
$2$ distinct eigenvalues in $\Fbar$ (the square roots of the roots of $t^2+\alpha t+\beta$).
Hence $\calB_s$ is $\overline{2}$-spec. In matrix terms, this shows that the space
$\calB_4(\F)$
of all matrices of the form
$$\begin{bmatrix}
A & S_2 \\
S_1 & -A^T
\end{bmatrix}, \quad \text{with $A \in \Mat_2(\F)$ and $(S_1,S_2)\in \Mats_2(\F)^2$,}$$
is $2$-spec. It can be proved that $\calB_4(\F)$ is not conjugated to $\mathfrak{sl}_2(\F) \vee \mathfrak{sl}_2(\F)$
(e.g.\ by proving that it is irreducible), but we will not deal with this issue here.

In a forthcoming article, we will extend Theorems \ref{theorem:maintheodim1starspec} and
\ref{theorem:maintheodim2spec} by giving a full classification of the optimal
$\overline{1}^\star$-spec subspaces and of the optimal $\overline{2}$-spec subspaces, thanks to a key result that will be proved in the present manuscript.
Here, we will also prove the easiest parts of this classification (oddly enough, among these parts is
the very special case $n=4$ for $\overline{2}$-spec subspaces!).
Let us state the future results, insisting that they will not be proved here
with the exception of Theorem \ref{theo:dim4}.

\begin{theo}
Let $n \geq 2$, and $\F$ be a field with characteristic $2$ and $|\F|>2$.
Let $\calS$ be an optimal $\overline{1}^\star$-spec subspace of $\Mat_n(\F)$.
Then there is a unique integer $k \in \lcro 0,n-2\rcro$ such that $\calS$ is similar to
$\NT_k(\F) \vee \mathfrak{sl}_2(\F) \vee \NT_{n-k-2}(\F)$.
\end{theo}

For $\overline{2}$-spec subspaces, the classification is more complicated, and not only due to the case $n=4$.

\begin{theo}
Let $n \geq 3$ differ from both $4$ and $6$, and $\F$ be a field with characteristic $2$ and $|\F|>2$.
Let $\calS$ be an optimal $\overline{2}$-spec subspace of $\Mat_n(\F)$.
Then there is a unique integer $k \in \lcro 0,n-2\rcro$ such that $\calS$ is similar to
$\F I_n\oplus (\NT_k(\F) \vee \mathfrak{sl}_2(\F) \vee \NT_{n-k-2}(\F))$.
\end{theo}

\begin{theo}\label{theo:dim4}
Let $\F$ be a field with characteristic $2$ such that $|\F|>2$.
Let $\calS$ be an optimal $2$-spec subspace of $\Mat_4(\F)$.
Then $\calS$ is similar to $\mathfrak{sl}_2(\F) \vee \mathfrak{sl}_2(\F)$ or to $\calB_4(\F)$.
\end{theo}

In \cite{dSPfeweigenvalues}, we formulated a conjecture on the case $n=4$, but this conjecture missed the example of $\calB_4(\F)$.
Note that Theorem \ref{theo:dim4} holds for $2$-spec subspaces, not only for $\overline{2}$-spec subspaces.

\begin{theo}
Let $\F$ be a field with characteristic $2$ such that $|\F|>2$.
Let $\calS$ be an optimal $\overline{2}$-spec subspace of $\Mat_6(\F)$.
Then exactly one of the following four situations occurs:
\begin{enumerate}[(i)]
\item There is an integer $k \in \lcro 1,3\rcro$ such that $\calS$ is similar to
$\F I_6\oplus (\NT_k(\F) \vee \mathfrak{sl}_2(\F) \vee \NT_{4-k}(\F))$, and in this case $k$ is uniquely determined by $\calS$.
\item There is an optimal $\overline{1}$-spec subspace $\calW$ of $\Mat_4(\F)$ such that
$\calS \simeq \mathfrak{sl}_2(\F) \vee \calW$, in which case the similarity class of $\calW$ is uniquely determined by $\calS$.
\item There is an optimal $\overline{1}$-spec subspace $\calW$ of $\Mat_4(\F)$ such that $\calS \simeq \calW \vee \mathfrak{sl}_2(\F)$,
in which case the similarity class of $\calW$ is uniquely determined by $\calS$.
\item The space $\calS$ is similar to the linear subspace of $\mathfrak{sl}_2(\F)\vee \mathfrak{sl}_2(\F) \vee \mathfrak{sl}_2(\F)$
consisting of its matrices with equal first and third diagonal blocks.
\end{enumerate}
\end{theo}

Again, case (iv) was missing from the conjectures formulated in \cite{dSPfeweigenvalues}.
Besides, the classification of the optimal $\overline{1}$-spec subspaces of $\Mat_4(\F)$ was achieved in \cite{dSPsoleeigenvaluecorr}.

\subsection{Why leave $\F_2$ out?}

Throughout this article, we consider only fields with characteristic $2$, but we systematically leave out those with cardinality $2$.
The main reason for this is that $\F_2$ behaves completely differently from the fields with greater cardinality with respect to the problems under consideration.

First of all, it is obvious that every element of an $\F_2$-algebra is $2$-spec and $1^\star$-spec, so for these properties there is simply no problem worth speaking of.
In contrast, the $\overline{2}$-spec and $\overline{1}^\star$-spec properties are not trivial.
Recall that every irreducible polynomial over a finite field is separable, and in particular if such a polynomial has degree more than $1$
then it has several roots in an algebraic closure of the ground field. Hence, an element of a finite-dimensional $\F_2$-algebra is $\overline{1}^\star$-spec if and only if its minimal polynomial is of the form $t^a (t-1)^b$. It follows that a square matrix with entries in $\F_2$
is $\overline{1}^\star$-spec if and only if it is triangularizable over $\F_2$.
Hence, an example of a large $\overline{1}^\star$-spec linear subspace of $\Mat_n(\F_2)$ is the linear subspace of all upper-triangular matrices, which has dimension
$\dbinom{n+1}{2}$: this dimension is substantially greater, for large values of $n$, than the bound $\dbinom{n}{2}+2$ from Theorem
\ref{theorem:maintheodim1starspec}, which illustrates that fields with cardinality $2$ require a completely different analysis than the other fields with characteristic $2$.
Moreover, the problem now coincides with the special case of fields with cardinality $2$ in the study of the so-called
\emph{weakly triangularizable} subspaces, which was recently initiated \cite{dSPtriangularizable}.
At the present, the known results on such spaces leave out fields with characteristic $2$ and small cardinality, and all the known techniques fail to
handle fields with cardinality $2$ with the exception of matrices of very small dimension (up to $n=3$).

The following conjecture seems reasonable, but so far all our attempts have failed to prove it:

\begin{conj}
Let $n \geq 2$.
The greatest possible dimension for a $\overline{1}^\star$-spec linear subspace of $\Mat_n(\F_2)$ is $\dbinom{n+1}{2}$.
\end{conj}

Moreover, a plausible conjecture is that the only spaces of greatest possible dimension are conjugates of joins of copies of
the spaces $\Mat_1(\F_2)$ and $\mathfrak{sl}_2(\F_2)$. We believe that, if true, proving it is a formidable challenge.

Let us finish with a discussion of $\overline{2}$-spec subspaces over $\F_2$. Over $\F_2$ there is a unique irreducible polynomial with degree $2$, namely $t^2+t+1$.
Hence, a matrix of $\Mat_n(\F_2)$ is $\overline{2}$-spec if and only if its characteristic polynomial is of either one of the forms $t^a(t-1)^b$ or $(t^2+t+1)^c$.
The latter case is possible only if $n$ is even. In particular, if $n$ is odd then a matrix of  $\Mat_n(\F_2)$ is $\overline{2}$-spec if and only if it is
$\overline{1}^\star$-spec, and hence the study of $\overline{2}$-spec subspaces is entirely reduced to the one of weakly triangularizable subspaces.
If $n$ is even then a $\overline{2}$-spec linear subspace of $\Mat_n(\F)$ can contain non-triangularizable matrices, which further complicates the problem.
We suspect however that this does not affect the value of the greatest possible dimension, and we suggest the following conjecture (but warn the reader that it is probably
extremely difficult to prove it):

\begin{conj}
Let $n \geq 3$.
The greatest possible dimension for a $\overline{2}$-spec linear subspace of $\Mat_n(\F_2)$ is $\dbinom{n+1}{2}$,
with equality attained only for $\overline{1}^\star$-spec subspaces.
\end{conj}

\subsection{Proof strategy}\label{section:strategy}

We now lay out the strategy for proving the results stated in the previous section.
We have seen that they are slightly different from the ones that hold for fields with characteristic other than $2$,
and it is critical to explain what fails in the proof method that was used in \cite{dSPfeweigenvalues} to deal with those fields.
The key strategy in \cite{dSPfeweigenvalues} was the so-called \emph{adapted vectors method}.
The essence of the adapted vectors method is:
\begin{itemize}
\item an inductive proof on the size of the matrices;
\item a very close look at the structure of the subset of rank $1$ elements in the matrix space $\calM$ under consideration.
\end{itemize}
Here, it will prove very useful to adopt a more geometric viewpoint, and we will think in terms of spaces of endomorphisms of a finite-dimensional vector space.
Now, let us see how the typical inductive proof works, for $1^\star$-spec subspaces.
So, we let $n \geq 3$ be an integer and assume that we have proved that every $1^\star$-spec subspace of endomorphisms of an $(n-1)$-dimensional vector
space has dimension at most $\binom{n-1}{2}+2$. We take an $n$-dimensional vector space $V$ and a $1^\star$-spec subspace $\calS$ of $\End(V)$.
Then, we take an arbitrary nonzero vector $x \in V \setminus \{0\}$, and we consider the linear subspace
$$\calS_x:=\{u \in \calS : u(x) \in \F x\}$$
of $\calS$. By the rank theorem, we have
$$\dim \calS \leq \dim (\calS_x)+(n-1).$$
Every endomorphism $u \in \calS_x$ induces an endomorphism $\overline{u}$ of the quotient space $V/\F x$,
obviously with at most one non-zero eigenvalue in $\F$,
and we recover a $1^\star$-spec subspace
$$\overline{\calS_x}:=\{\overline{u} \mid u \in \calS_x\} \subseteq \End(V/\F x).$$
Now, by the rank theorem we find
$$\dim \calS_x=\dim \overline{\calS_x}+\dim \{u \in \calS: \im u \subseteq \F x\}.$$
Hence by induction
\begin{align*}
\dim \calS & \leq \dim \{u \in \calS : \im u \subseteq \F x\}+(n-1)+\dbinom{n-1}{2}+2\\
& \leq \dim \{u \in \calS : \im u \subseteq \F x\}+\dbinom{n}{2}+2.
\end{align*}

To move forward efficiently, some additional notation is in order:

\begin{Not}
Let $V$ be a finite-dimensional vector space, with dual vector space denoted by $V^\star$.
For a vector $y \in V$ and a linear form $\varphi \in V^\star$, we set
$$\varphi \otimes y : x \in V \mapsto \varphi(x)\,y \in V,$$
which is an endomorphism of $V$ with trace $\varphi(y)$.
\end{Not}

Note that most authors would rather use the notation $y \otimes \varphi$, but our notation is more convenient because it is respectful of the natural order in the expression for $(\varphi \otimes y)(x)$.

\begin{Not}
For a vector $x \in V \setminus \{0\}$, we also denote by $x^\bot$ the set of all linear forms $\varphi$ on $V$ such that $\varphi(x)=0$:
it is a linear hyperplane of $V^\star$.
\end{Not}

For an arbitrary vector $x \in V \setminus \{0\}$, the set $x^\bot \otimes x:=\{\varphi \otimes x \mid \varphi \in x^\bot\}$
is actually the linear subspace of $\End(V)$ that consists of all the endomorphisms of $V$ with trace zero and range included in $\F x$.
We can now formulate the key definition:

\begin{Def}
Let $\calS$ be a linear subspace of $\End(V)$, where $V$ is a finite-dimensional vector space.
A nonzero vector $x \in V \setminus \{0\}$ is called \textbf{$\calS$-adapted} whenever
$\calS \cap (x^\bot \otimes x)=\{0\}$, i.e., $\calS$ contains no operator with trace zero and range $\F x$.
\end{Def}

Now, assume for a moment that $x$ is $\calS$-adapted.
Then there are two cases:
\begin{itemize}
\item Either $\calS \cap \Hom(V,\F x)=\{0\}$, in which case we directly obtain
$\dim \calS \leq \dbinom{n}{2}+2$ by induction;
\item Or $\calS \cap \Hom(V,\F x)$ has dimension $1$, and it contains in particular a nonzero element $\varphi \otimes x$, which must then satisfy $\varphi(x) \neq 0$.
Let then $u \in \calS_x$. We have $u(x)=\alpha x$ for some $\alpha \in \F$, and it follows that
$\overline{u}$ has at most one nonzero eigenvalue that differs from $\alpha$.
For all $\lambda \in \F$, we apply this principle to $v:= u+\lambda\, \varphi \otimes x$, which satisfies $\overline{v}=\overline{u}$
and $v(x)=(\alpha+\lambda \varphi(x))\, x$, and we deduce that $\overline{u}$ has no nonzero eigenvalue in $\F$.
Then we directly apply Theorem \ref{theo:0starspec} to see that $\dim \overline{\calS} \leq \dbinom{n-1}{2}$,
and this time around we deduce from the rank theorem that $\dim \calS \leq \dbinom{n}{2}+1$.
\end{itemize}

The key to the proof of the main theorems in \cite{dSPfeweigenvalues} was
that, in the context of that article (i.e., fields with characteristic other than $2$) every $1^\star$-spec subspace of endomorphisms
of $V$ has an adapted vector provided that $\dim V \geq 2$ (and ditto for every $2$-spec subspace provided that $\dim V \geq 3$).
Unfortunately, this surely fails in the present context, as is made clear not only by the example of $\mathfrak{sl}_2(\F)$,
but also by the one of $\Mat_{n-2}(\F) \vee \mathfrak{sl}_2(\F)$: indeed, in the latter case it is quite easy to prove that
$\calS$ contains a trace-zero matrix with range $\F x$ for \emph{all} $x \in \F^n \setminus \{0\}$.

Here, instead of trying a completely different method, we will face this difficulty and try to get around it, in two main ways.
The first way, which is inspired by recent work on spaces of triangularizable matrices \cite{dSPtriangularizable}, consists in using the
supposed lack of an adapted vector to our \emph{advantage}: indeed, the lack of an $\calS$-adapted vector can be translated
into a property of intransitivity on the part of the trace-dual orthogonal space
$$\calS^\bot:=\{v \in \End(V) : \; \tr(uv)=0\},$$
of which we recall the properties $\dim \calS^\bot=n^2-\dim \calS$ and $\calS=(\calS^\bot)^\bot$.
Given $x \in V \setminus \{0\}$, it is classical (see e.g.\ lemma 3.3 in \cite{dSPtriangularizable}) that
\begin{equation}\label{eq:transrank}
\dim (\calS^\bot x)=\dim V- \dim (\calS \cap (V^\star \otimes x)),
\end{equation}
and hence having $x$ non-$\calS$-adapted readily yields $\dim (\calS^\bot x) < \dim V$.
Hence, if $\calS$ has no adapted vector then $\calS^\bot$ is intransitive in the following meaning of the word:

\begin{Def}
Let $\calT$ be a linear subspace of $\Hom(U,V)$, where $U$ and $V$ are finite-dimensional vector spaces.
We say that $\calT$ is \textbf{intransitive} whenever $\calT x \neq V$ for all $x \in U$.

In any case, we define the \textbf{transitive rank} of $\calT$ as $\max\{\dim (\calT x) \mid x \in U\} \in \lcro 0,\dim V\rcro$
and denote it by $\trk(\calT)$.

We say that $\calT$ is \textbf{primitively intransitive} when $V$ is nonzero, $\calT$ is intransitive and there is no nonzero linear subspace $H$ of
$V$ such that, for the standard projection $\pi : V \twoheadrightarrow V/H$, the operator space $\pi \calT$ is intransitive.
\end{Def}

Provided that the cardinality of $\F$ is large enough, a theorem of Atkinson on (semi)-primitive spaces of bounded rank matrices can be translated into
a theorem on primitively intransitive operator spaces, which can help us understand the structure of $\calS^\bot$.
The formulation (and a proof) of this result can be found in \cite{dSPsemilin} (there, combine point (b) of proposition 4.5
with the case $d=1$ in theorem 1.20).

\begin{theo}[Atkinson's theorem for intransitive operator spaces]
Let $U$ and $V$ be finite-dimensional vector spaces, with $n:=\dim V>0$, and
let $\calT$ be a primitively intransitive linear subspace of $\Hom(U,V)$.
Assume that $|\F| \geq n$.
Then:
\begin{enumerate}[(a)]
\item One has $\dim \calT \leq \dbinom{n}{2}$, and even $\dim \calT \leq \dbinom{n-1}{2}$ if $\trk(\calT) <n-1$.
\item If $\dim \calT \geq \dbinom{n}{2}-(n-3)$ then there exists a bilinear form $b : U \times V \rightarrow \F$
that is right-nondegenerate and such that $\forall x\in U, \; \forall f \in \calT, \; b(x,f(x))=0$:
we say that $b$ is an \textbf{alternator} of $\calT$.
\end{enumerate}
\end{theo}

Recall that a bilinear form $b : U \times V \rightarrow \F$ is right-nondegenerate if and only if $b(-,y) \neq 0$ for all $y \in V \setminus \{0\}$.

It might seem odd to apply Atkinson's theorem to $\calS^\bot$, because using point (a) results in giving a \emph{lower} bound for $\calS$, not an upper one!
But now comes the trick: we need only consider the case where $\calS$ has its dimension just one unit larger than the expected upper-bound $\dbinom{n}{2}+2$, as if larger spaces exist then we can extract such a subspace and try to find a contradiction;
in the critical case then, the dimension of $\calS^\bot$ is exactly $n^2-\dbinom{n}{2}-3$, that is $\dbinom{n}{2}+(n-3)$, which is larger than the bound from point (a) of Atkinson's theorem unless $n \leq 3$. Hence, and if $\calS^\bot$ is primitively intransitive of course, either we obtain a contradiction or we can use point (b) to gain a sharp understanding of the structure of $\calS^\bot$ for $n=3$.
There are two difficulties of course with that strategy:
\begin{itemize}
\item[Q1.] What to do if $\calS^\bot$ is not primitively intransitive?
\item[Q2.] What if $|\F|<n$?
\end{itemize}
The first point is entirely manageable, and it will systematically involve the following construction:
if $\calS^\bot$ is not primitively intransitive then we can take a proper subspace $W \subset V$ of greatest possible dimension
such that, for the canonical projection $\pi : V \twoheadrightarrow V/W$, the space $\pi \calS^\bot$ is intransitive:
we say that $W$ is an \textbf{intransitivity veil} of $\calS^\bot$.
\label{def:intransitivityveil}
Then $\pi \calS^\bot$ is primitively intransitive due to the maximality of $W$, and one has
$$\dim \calS^\bot=\dim(\pi \calS^\bot) +\dim(\calS^\bot \cap \Hom(V,W))$$
by the rank theorem. Further still, we will be able to apply Atkinson's theorem to $\pi \calS^\bot$, while the
dimension of the space
$$(\calS^\bot)^W:=\calS^\bot \cap \Hom(V,W)$$
is critically connected to the one of the space of restrictions
$$\calS_{|W}:=\{u_{|W} \mid u \in \calS\} \subseteq \Hom(W,V)$$
through the identity
\begin{equation}\label{eq:dualrestriction}
\dim \left((\calS^\bot)^W\right)+\dim(\calS_{|W})=(\dim V)(\dim W)
\end{equation}
(see, e.g., lemma 3.4 in \cite{dSPtriangularizable}).
From there, and thanks to an induction hypothesis, it will be possible to find a direct contradiction with the assumptions on the dimension of $\calS$.

However, the answer to Q2 will probably be a disappointment to the reader: we must accept that the strategy completely fails in that case!

To obtain the results for larger values of $n$ and whatever the field under consideration (but still with characteristic $2$ and more than $2$ elements),
we have to use a completely different strategy. We will try to classify, in a direct way, the
situations where $\calS$ has no adapted vector. This involves the following definition:

\begin{Def}
Let $V$ be a vector space with finite dimension $n \geq 2$. A linear subspace $\calS$ of $\End(V)$ is called a \textbf{hurdle} when
it satisfies the following equivalent conditions:
\begin{enumerate}[(i)]
\item There exists a linear subspace $G$ of $V$ with codimension $2$ and such that $\calS$ contains \emph{all} the trace-zero endomorphisms $u$ of $V$
such that $G \subset \Ker u$.
\item There exists a $2$-dimensional subspace $P$ of $V^\star$ such that $\calS$ contains all the tensors $\varphi \otimes y$
with $\varphi\in P$ and $y \in V$ such that $\varphi(y)=0$.
\item In some basis of $V$, the matrix space $\calM$ that represents $\calS$ includes $\{0_{n-2}\} \vee \mathfrak{sl}_2(\F)$.
\end{enumerate}
\end{Def}

It is clear from  (ii) that a hurdle has no $\calS$-adapted vector, and the key is to prove precisely
that among the $1^\star$-spec subspaces, only hurdles have that property.

However, this will turn out to be very difficult, requiring us to considerably refine the method devised in \cite{dSPfeweigenvalues}.
Until this day, one of the distinctive features of the adapted vectors method was to completely separate the search for adapted vectors
from the considerations on the dimension of the matrix space under consideration: one performed an inductive proof to obtain the existence of an adapted vector, and once a general existence theorem was obtained, one applied it (again, inductively) to prove the dimension bound.
Here, we will completely break away from this strategy, and we will prove all the properties \emph{simultaneously} by induction.
This is made necessary by another key point in the proof of the existence of adapted vectors from \cite{dSPfeweigenvalues}, which is the final argument: there, the consequences of having too few adapted vectors was the existence of a subspace of $\calS$ spanned by matrices of rank $1$ and trace $0$, and with dimension greater than or equal to $n \left(\lfloor \frac{n}{2}\rfloor+1\right)$;
thanks to the assumption that $\car(\F) \neq 2$, one could then observe that any two matrices in the subspace must be trace-orthogonal (lemma 2.3 there) and obtain a contradiction with standard orthogonality results. Here, there is no analogue of lemma 2.3 from \cite{dSPfeweigenvalues} because of the characteristic $2$ assumption (precisely, since every matrix with rank $2$ and trace $0$ has at most one eigenvalue in $\overline{\F}$). Here, we rescue the overall idea that having too few adapted vectors either leads to $\calS$ being a hurdle or $\calS$ having its dimension too large, but we must use another way of bounding the dimension of $\calS$ beforehand. The key new insight is to obtain ever smaller upper-bounds on $\dim(\calS)$, which involves the following weak version of an $\calS$-adapted vector:

\begin{Def}
Let $\calS$ be a linear subspace of $\End(V)$, where $V$ is a finite-dimensional vector space.
A nonzero vector $x \in V \setminus \{0\}$ is called \textbf{weakly-$\calS$-adapted} whenever
$\dim(\calS \cap (x^\bot \otimes x)) \leq 1$.
\end{Def}

The proof will then essentially run as follows:
\begin{itemize}
\item By taking an arbitrary nonzero vector, we can use the induction hypothesis to immediately obtain the
upper-bound $\dim \calS \leq \dbinom{n}{2}+2+(n-1)$, which is already quite good (yet insufficient);
\item Then we will show (which uses a reinforced inductive assumption) that we can \emph{always} guarantee that $\calS$ has (many) weakly-$\calS$-adapted vectors;
\item Using a weakly $\calS$-adapted vector, we will reuse the induction hypothesis to obtain the improved inequality
$\dim \calS \leq \dbinom{n}{2}+3$;
\item This inequality will then be sufficient to prove (again thanks to a reinforced inductive assumption) that
either $\calS$ has (many) adapted vectors or it is a hurdle. When it is a hurdle, we will obtain the
inequality $\dim \calS \leq \dbinom{n}{2}+2$ directly, whereas in the other case the conclusion follows by induction.
\end{itemize}

At this point, we note that we could completely avoid using Atkinson's theorem, yet it is very powerful to handle the situation of matrices of small size
(in which the requirements on the cardinality of $\F$ are fulfilled).

So far, we have only discussed $1^\star$-spec subspaces. The strategy is very similar for $2$-spec subspaces, with only minor differences.
One directly considers
$$\calS_{x,0}:=\{u \in \calS : u(x)=0\},$$
and the induced space
$$\overline{\calS_{x,0}}:=\{\overline{u} \mid u \in \calS_{x,0}\} \subseteq \End(V/\F x).$$
The rank theorem readily yields
$$\dim \calS_x =\dim (\calS x)+\dim \overline{\calS_{x,0}}+\dim (\calS \cap (x^\bot \otimes x)) \leq n+\dim \overline{\calS_{x,0}}+\dim (\calS \cap (x^\bot \otimes x)).$$
Assume now that $\calS$ is $2$-spec. Then $\calS_{x,0}$ is $1^\star$-spec (because,
for every operator in $\calS_{x,0}$, the vector $x$ is obviously an eigenvector for the eigenvalue $0$). If $x$ is $\calS$-adapted, and provided Theorem \ref{theorem:maintheodim1starspec} holds true, then
$$\dim \calS \leq \dbinom{n}{2}+3.$$
If $x$ is only weakly-$\calS$-adapted then we have the rougher estimate
$$\dim \calS \leq \dbinom{n}{2}+4.$$
One essential difference, which the above exemplifies, lies in the fact that the proofs for $2$-spec subspaces are not inductive: rather, they rely upon the validity of the results
for $1^\star$-spec spaces.

\subsection{Structure of the article}

The article is organized as follows.

We have gathered a series of basic lemmas in Section \ref{section:basiclemmas}: these lemmas deal with
the way of localizing the non-adapted vectors for a matrix space, and some structural results on hurdles: they are used repeatedly in the article.

Section \ref{section:largefields} features a proof of the main theorems for fields with large cardinality, by using Atkinson's theorem on intransitive operator spaces.

The remainder of the article deals with the generalization to all fields with characteristic $2$ and more than $2$ elements.
Some preparatory work is done in Section \ref{section:preparationforlastpart}, which mainly deals with giving a condition for a polynomial in several variables to vanish once
we know that it vanishes outside of a union of certain linear subspaces. The last section (Section \ref{section:anyfield}), by far the longest and the most intricate of all,
concludes the proof of the main theorems by eliminating all restrictions of cardinality (except of course the requirement that $|\F|>2$),
thanks in part to the fact that the situation for small matrices is already covered in Section \ref{section:largefields}.

\section{Some lemmas}\label{section:basiclemmas}

Here, we will collect various lemmas that deal with situations in which we have an operator space
that contains a large subset of rank $1$ operators. Their proofs require the use of cyclic matrices, of which we will recall several classical properties.
We start with a short review of duality techniques, to be used throughout the article.

\subsection{On dualities}

We will frequently have to use duality arguments. Let $V$ be a finite-dimensional vector space.
For $u \in \End(V)$, we consider its transposed endomorphism
$$u^t : \varphi \in V^\star \mapsto \varphi \circ u \in V^\star.$$
Classically, $u$ and $u^t$ have the same characteristic polynomial. It follows that
the space $\calS^t:=\{u^t \mid u \in \calS\}$ is $k$-spec (respectively, $\overline{k}$-spec and so on)
if and only if so is $\calS$.
Denote by
$$\mathfrak{i}_V : x \in V \mapsto (\varphi \mapsto \varphi(x)) \in V^{\star\star}$$
the biduality isomorphism.
Throughout, it will be very useful to observe that $(\varphi \otimes x)^t=\mathfrak{i}_V(x) \otimes \varphi$ for all $\varphi \in V^\star$ and all $x \in V$.

For a linear subspace $F$ of $V$, we set
$$F^\ortho:=\{\varphi \in V^\star : \forall x \in F, \; \varphi(x)=0\},$$
whereas, for a linear subspace $G$ of $V^\star$ we set
$${}^\ortho G:=\{x \in V : \forall \varphi \in G, \; \varphi(x)=0\}.$$
We recall that $\dim(F^\ortho)=\dim V-\dim F$, $\dim ({}^\ortho G)=\dim V-\dim G$, $G=({}^\ortho G)^\ortho$ and $F={}^\ortho (F^\ortho)$.

Next, we consider some standard results on trace orthogonality. The first one generalizes 
identity \eqref{eq:dualrestriction} on page \pageref{eq:dualrestriction}.
In both, we need the general notion of the trace-orthogonal of an operator space. Given finite-dimensional vector spaces $U$ and $V$
and a linear subspace $\calS$ of $\Hom(U,V)$, we set
$$\calS^\bot:=\{v \in \Hom(V,U) : \tr(vu)=0\}.$$

\begin{lemma}[First Trace Orthogonality Lemma]\label{lemma:traceortho1}
Let $U$ and $V$ be finite-dimensional vector spaces and $V_0$ be a linear subspace of $V$. Let $\calS \subseteq \Hom(U,V)$ be a linear subspace.
Then
$$\dim(\calS \cap \Hom(U,V_0))+\dim\{v_{|V_0} \mid v \in \calS^\bot\}=(\dim U)(\dim V_0).$$
\end{lemma}

See lemma 3.4 in \cite{dSPtriangularizable} and the proof therein. The next result is a dual variation of the previous one:

\begin{lemma}[Second Trace Orthogonality Lemma]\label{lemma:traceortho2}
Let $U$ and $V$ be finite-dimensional vector spaces and $U_0$ be a linear subspace of $U$. Let $\calS \subseteq \Hom(U,V)$ be a linear subspace.
Denote by $\pi : U \twoheadrightarrow U/U_0$ the standard projection. Set
$$\calS_{U_0}:=\{u \in \calS : \; \forall x \in U_0, \; u(x)=0\}.$$
Then
$$\dim \calS_{U_0}+\dim (\pi \calS^\bot)=(\dim V)(\dim U-\dim U_0).$$
\end{lemma}

\begin{proof}
Although a direct proof is possible, we will simply deduce this result from the previous one by a duality argument.

We consider the transposed space $\calS^t \subseteq \Hom(V^\star,U^\star)$ and we note, since the trace of an endomorphism is invariant under transposition,
that $(\calS^t)^\bot=(\calS^\bot)^t$. We also note that $(\calS_{U_0})^t=\calS^t \cap \Hom(V^\star,U_0^\circ)$
and that $(\pi \calS^\bot)^t=(\calS^\bot)^t \pi^t=(\calS^t)^\bot \pi^t$.
The range of $\pi^t$ is exactly $U_0^\circ$, and $\pi^t$ is injective, so the space
$(\calS^t)^\bot \pi^t$ is naturally isomorphic to $\{v_{|U_0^\circ} \mid v \in (\calS^t)^\bot\}$.
Applying Lemma \ref{lemma:traceortho1} to $\calS^t$ with the linear subspace $V_0:=U_0^\circ$ of $U^\star$,
we obtain
$$\dim (\calS^t \cap \Hom(V^\star,U_0^\circ))+\dim \{v_{|U_0^\circ} \mid v \in (\calS^t)^\bot\}=(\dim V^\star)(\dim U_0^\circ).$$
The left-hand side of the equality equals
$$\dim (\calS_{U_0})^t+\dim ((\calS^t)^\bot \pi^t)=\dim \calS_{U_0}+ \dim (\pi \calS^\bot),$$
while the right-hand side equals $(\dim V)(\dim U-\dim U_0)$.
\end{proof}

\subsection{On cyclic matrices}

The \textbf{companion matrix} of a monic polynomial $r(t)=t^n-\underset{k=0}{\overset{n-1}{\sum}} a_k t^k$ of $\F[t]$
is defined as
$$C(r):=\begin{bmatrix}
0 &   & & (0) & a_0 \\
1 & 0 & &   & a_1 \\
0 & \ddots & \ddots & & \vdots \\
\vdots & \ddots & & 0 & a_{n-2} \\
(0) & \cdots & 0 &  1 & a_{n-1}
\end{bmatrix}\in \Mat_n(\F).$$
Classically, its minimal polynomial and characteristic polynomial equal $r(t)$.
We also recall that the \textbf{trace} of $r(t)$ is then defined as $a_{n-1}$ (the opposite of its coefficient on $t^{n-1}$)
and denoted by $\tr(r(t))$.

The next lemma will be critical to us. It is very classical, and can be found, e.g., as lemma 11 in \cite{dSP3idempotent}.

\begin{lemma}[Choice Lemma]\label{label:choicelemma}
Let $M=(m_{i,j})_{i,j} \in \Mat_n(\F)$ be a regular Hessenberg matrix, i.e.,
$m_{i,j}=0$ for all $i,j$ such that $i>j+1$, and $m_{j+1,j} \neq 0$ for all $j \in \lcro 1,n-1\rcro$.
Let $r(t) \in \F[t]$ be a monic polynomial of degree $n$ such that $\tr (r(t))=\tr M$.
Let $p \in \lcro 1,n-1\rcro$. Then there exists a matrix $R \in \Mat_{p,n-p}(\F)$ such that
$$M+\begin{bmatrix}
[0]_{p \times p} & R \\
[0]_{(n-p) \times p} & [0]_{(n-p) \times (n-p)}
\end{bmatrix}$$
has characteristic polynomial $r(t)$.
\end{lemma}

\subsection{The first confinement lemma}

\begin{lemma}[First Confinement Lemma]
Let $\calS$ be a $2$-spec linear subspace of $\End(V)$ for some $n$-dimensional vector space $V$, with $n \geq 3$.
Assume that $\varphi \otimes V \subseteq \calS$ for some $\varphi \in V^\star \setminus \{0\}$.
Then all the non-$\calS$-adapted vectors of $V$ belong to the linear hyperplane $\Ker \varphi$.
\end{lemma}

\begin{lemma}[First Confinement Lemma, dual version]
Let $\calS$ be a $2$-spec linear subspace of $\End(V)$ for some $n$-dimensional vector space $V$, with $n \geq 3$.
Assume that $V^\star \otimes x \subseteq \calS$ for some $x \in V\setminus \{0\}$.
Then every rank $1$ and trace zero operator in $\calS$ must vanish at $x$.
\end{lemma}

\begin{proof}[Proof of the dual version]
Assume on the contrary that $\calS$ contains a rank $1$ and trace zero operator $u=\varphi \otimes y$ such that $\varphi(x) \neq 0$.
Since $u$ has trace zero, we have $\varphi(y)=0$, and hence $x$ and $y$ are linearly independent, to the effect that $x$ and $u(x)$ are linearly independent.
Moreover $u^2(x)=\alpha u(x)$ for some $\alpha \in \F$. We extend the family $(x,u(x))$ into a whole basis
of $V$. The matrix of $u$ in this basis now looks as
$$M=\begin{bmatrix}
C(p) & [?]_{2 \times (n-2)} \\
[0]_{(n-2) \times 2} & [0]_{(n-2) \times (n-2)}
\end{bmatrix} \quad \text{where $p(t):=t^2-\alpha t$.}$$
Let $\psi \in V^\star$. Then $u+\psi \otimes x$, which belongs to $\calS$, has rank at most $2$, and hence has $0$ in its spectrum (recall that $n \geq 3$).
Hence the endomorphism of $\Vect(x,u(x))$ induced by $u+\psi \otimes x$ has at most one non-zero eigenvalue.
In varying $\psi$, we deduce that the matrix $\begin{bmatrix}
a & b \\
1 & \alpha
\end{bmatrix}$ has at most one nonzero eigenvalue for all $(a,b)\in \F^2$.
Choose distinct elements $t_1$ and $t_2$ in $\F \setminus \{0\}$. Setting $a:=t_1+t_2-\alpha$, we see that $b$ can be adjusted so that the characteristic polynomial
of $\begin{bmatrix}
a & b \\
1 & \alpha
\end{bmatrix}$ equals $(t-t_1)(t-t_2)$, resulting in a matrix with two distinct nonzero eigenvalues in $\F$. This contradiction completes the proof.
\end{proof}

\begin{proof}[Proof of the First Confinement Lemma]
The assumptions translate into $V^{\star \star} \otimes \varphi \subseteq \calS^t$, and $\calS^t$ is a $2$-spec linear subspace of $\End(V^\star)$.

Let $x \in V \setminus \{0\}$ be non-$\calS$-adapted. Then there exists a nonzero linear form $\psi$ such that $\psi \otimes x \in \calS$ and $\psi(x)=0$.
Hence $\mathfrak{i}_V(x) \otimes \psi$ is a trace zero and rank $1$ operator in $\calS^t$.
By the dual version of the First Confinement Lemma, the operator $\mathfrak{i}_V(x) \otimes \psi$ must then vanish at $\varphi$,
to the effect that $\varphi(x)=0$.
\end{proof}

\subsection{On hurdles}

\begin{lemma}[Splitting Lemma for Hurdles]\label{lemma:splittinghurdles}
Let $n \geq 3$ be an integer, $V$ be an $n$-dimensional vector space,
$G$ be a linear subspace of $V$ with codimension $2$, and
$\calS$ be a linear subspace of $\End(V)$. Assume that
$\varphi \otimes x \in \calS$ for all $\varphi \in G^o$ and all $x \in V$ such that $\varphi(x)=0$.
Assume finally that $\calS$ is $1^\star$-spec (respectively, $2$-spec).
Then, for every $u \in \calS$:
\begin{enumerate}[(a)]
\item $u$ leaves $G$ invariant;
\item The induced endomorphism $u_G$ has no nonzero eigenvalue in $\F$ (respectively, at most one eigenvalue in $\F$);
\item The induced endomorphism $u_G$ has no eigenvalue in $\F$ if the induced endomorphism $u^{V/G}$ of $V/G$ has its trace nonzero;
\item The induced endomorphism $u^{V/G}$ of $V/G$ has its trace zero whenever $u_G=0$.
\end{enumerate}
\end{lemma}

\begin{lemma}[Splitting Lemma for Hurdles (dual version)]
Let $n \geq 3$ be an integer, $V$ be an $n$-dimensional vector space,
$P$ be a $2$-dimensional linear subspace of $V$, and
$\calS$ be a $2$-spec linear subspace of $\End(V)$. Assume that
$\varphi \otimes x \in \calS$ for all $\varphi \in V^\star$ and all $x \in P$ such that $\varphi(x)=0$.
Then $P$ is $\calS$-invariant.
\end{lemma}

\begin{proof}[Proof of the dual version]
Let $u \in \calS$.
We assume first that there is a vector $x \in P$ such that $x,u(x),u^2(x)$ are linearly independent.
Then, we consider the invariant subspace $W:=\F[u](x)$, which has a basis of the form $(x,u(x),\dots,u^{d-1}(x))$ for some integer $d \geq 3$, and the matrix of $u_W$ in that basis
is the companion matrix of some polynomial $p \in \F[t]$ (monic with degree $d$).
When $\varphi$ ranges over $x^\bot$, the matrices of $(u+\varphi \otimes x)_W$ in the basis
$(x,u(x),\dots,u^{d-1}(x))$ range over all the matrices of the form
$C(p)+\begin{bmatrix}
0 & [?]_{1 \times (d-1)} \\
[0]_{(d-1) \times 1} & [0]_{(d-1) \times (d-1)}
\end{bmatrix}$. Hence, by the Choice Lemma, their characteristic polynomials range over all the monic polynomials of degree $d$
with trace $\tr(p)$. To obtain a contradiction, it suffices to observe that at least one of these polynomials has at least three distinct roots in $\F$,
which is easy: we take three distinct elements $t_1,t_2,t_3$ in $\F$, and we take the polynomial $q:=(t-\tr(p)+t_1+t_2+t_3)t^{d-4}(t-t_1)(t-t_2)(t-t_3)$
if $d \geq 4$, whereas if $d=3$ we take $(t-t_1+\alpha)(t-t_2+\alpha)(t-t_3+\alpha)$ where $\alpha=\frac{1}{3}\left((t_1+t_2+t_3)-\tr(p)\right)$.

Hence $x,u(x),u^2(x)$ are linearly dependent for all $x \in P$, and this holds whatever the choice of $u$.
Now, consider the projection $\pi : V \twoheadrightarrow V/P$ and the composite $\pi \circ u_{|P} : P \rightarrow V/P$.
Assume first that $\pi \circ u_{|P}$ has rank $2$. Hence, by taking a basis $(x,y)$ of $P$, we find that $x,y,u(x),u(y)$ are linearly independent.
Then by the above $u^2(x)=\alpha u(x)+\lambda x$ and $u^2(y)=\beta u(y)+\mu y$ for some scalars $\alpha,\beta,\lambda,\mu$.
We choose $\gamma \in \F$ such that $\gamma \neq 0$ and $\gamma \neq \alpha^2-\alpha \beta$, to the effect that the polynomials $t^2-\alpha t$ and $t^2-\beta t-\gamma$
are relatively prime.
Then, we can adjust $\varphi$ and $\psi$ in $V^\star$ so that both $\varphi$ and $\psi$ vanish at $x$ and $y$, while requiring that
$\varphi(u(x))=-\lambda$, $\varphi(u(y))=0$, $\psi(u(x))=0$ and $\psi(u(y))=\gamma-\mu$. Hence $v:=u+\varphi\otimes x+\psi \otimes y$
satisfies $v(x)=u(x)$, $v^2(x)=\alpha u(x)=\alpha v(x)$, $v(y)=u(y)$ and $v^2(y)=\beta u(y)+\gamma y=\beta v(y)+\gamma y$.
Then classically, since $t^2-\alpha t$ and $t^2-\beta t-\gamma$ are relatively prime, the vector $x+y$ has minimal polynomial
$(t^2-\alpha t)(t^2-\beta t-\gamma)$ with respect to $v$, and in particular $x+y,v(x+y),v^2(x+y)$ are linearly independent.
This contradicts the previous step of proof since $v \in \calS$.

We deduce that $\pi \circ u_{|P}$ has rank at most $1$, which yields a vector $x \in P \setminus \{0\}$ such that $u(x) \in P$.
Assume finally that $\pi \circ u_{|P}$ has rank $1$. If $x,u(x)$ were linearly independent, then as $u^2(x) \in \Vect(x,u(x))$
we would deduce that $P$ is invariant under $u$, contradicting the fact that $\pi \circ u_{|P}$ has rank $1$.
Hence $u(x)=\alpha x$ for some $\alpha \in \F$. Then we extend $x$ into a basis $(x,y)$ of $P$.
Take finally an arbitrary linear form $\varphi$ on $V$ such that $\varphi(y)=0$ and $\varphi(x)=1$.
Then $v:=u+\varphi \otimes y$ belongs to $\calS$, $\pi \circ v_{|P}=\pi \circ u_{|P}$ has rank $1$ and $v(x) \in P$.
Applying the previous result to $v$, we deduce that $v(x) \in \F x$, and hence $y=v(x)-u(x) \in \F x$.
This is absurd, and we conclude that $\pi \circ u_{|P}=0$, which means that $P$ is invariant under $u$.
\end{proof}

\begin{proof}[Proof of the Splitting Lemma for Hurdles]
Applying the dual version to $\calS^t$ yields that the latter leaves $G^\circ$ invariant, and hence
$\calS$ leaves $G$ invariant.
Now, we use matrix representations to simplify the discourse: we represent $\calS$ by a matrix space $\calM$ in a basis
of $V$ whose first $n-2$ vectors constitute a basis of $G$.
Hence every matrix in $\calM$ has the form
$$M=\begin{bmatrix}
K(M) & [?]_{(n-2) \times 2} \\
[0]_{2 \times (n-2)} & s(M)
\end{bmatrix} \quad \text{with $K(M) \in \Mat_{n-2}(\F)$ and $s(M) \in \Mat_2(\F)$,}$$
with $K(M)$ and $s(M)$ representing, respectively, the induced endomorphisms $u_G$ and $u^{V/G}$.
The starting assumptions on $\calS$ yield that $\calM \supseteq \{0_{n-2}\} \vee \mathfrak{sl}_2(\F)$.

Let now $M \in \calM$.
We distinguish between two cases :
\begin{itemize}
\item Assume first that $\tr(s(M))=0$. Then, for every $\lambda \in \F$ we can add an appropriate matrix of $\{0_{n-2}\} \vee \mathfrak{sl}_2(\F)$ to $M$ and obtain that
$\calM$ contains
$\begin{bmatrix}
K(M) & [0]_{(n-2) \times 2} \\
[0]_{2 \times (n-2)} & \lambda I_2
\end{bmatrix}$, which yields that $K(M)$ has no nonzero eigenvalue in $\F$ (respectively, at most one eigenvalue in $\F$).

\item Assume next that $\tr(s(M)) \neq 0$, and set $\alpha:=\tr(s(M))$. Assume also that $K(M)$ has an eigenvalue $b$ in $\F$.
Choose $\lambda \in \F \setminus \{b,\alpha+b\}$,
Then, by adding an appropriate matrix of
$\{0_{n-2}\} \vee \mathfrak{sl}_2(\F)$, we gather that $\calM$ contains
$\begin{bmatrix}
K(M) & [0]_{(n-2) \times 2} \\
[0]_{2 \times (n-2)} & D
\end{bmatrix}$
where $D=\begin{bmatrix}
\alpha+\lambda & 0 \\
0 & \lambda
\end{bmatrix}$. Clearly, this matrix has at least three distinct eigenvalues in $\F$, which contradicts the assumption that $\calM$
is $2$-spec. This proves point (c), and in particular it yields point (d) (since the zero matrix of $\Mat_{n-2}(\F)$ has an eigenvalue in $\F$).
\end{itemize}
\end{proof}

For $\overline{1}^\star$-spec and $\overline{2}$-spec subspaces, the adaptation of the proof is straightforward, and one obtains the following result:

\begin{lemma}[Splitting Lemma for Hurdles (bar version)]
Let $n \geq 3$ be an integer, $V$ be an $n$-dimensional vector space,
$G$ be a linear subspace of $V$ with codimension $2$, and
$\calS$ be a linear subspace of $\End(V)$. Assume that
$\varphi \otimes x \in \calS$ for all $\varphi \in G^o$ and all $x \in V$ such that $\varphi(x)=0$.
Assume finally that $\calS$ is $\overline{1}^\star$-spec (respectively, $\overline{2}$-spec).
Then, for every $u \in \calS$:
\begin{enumerate}[(i)]
\item $u$ leaves $G$ invariant;
\item The induced endomorphism $u_G$ is nilpotent (respectively, has exactly one eigenvalue in $\Fbar$);
\item The induced endomorphism $u^{V/G}$ of $V/G$ has its trace zero.
\end{enumerate}
\end{lemma}

Here the possibility of having $\tr(u^{V/G}) \neq 0$ is ruled out because it would imply that $u_G$ has no eigenvalue in $\Fbar$.
This lemma will not be used in the present article, but will be required in a future analysis of the optimal spaces.

We finish with consequences of these lemmas on bounding the dimension of hurdles.

\begin{lemma}[Dimension inequality for hurdles]\label{lemma:ineqdimhurdles}
Let $\calS \subseteq \End(V)$ be a hurdle, with $n:=\dim V \geq 3$.
\begin{enumerate}[(a)]
\item If $\calS$ is $1^\star$-spec then $\dim \calS \leq \dbinom{n}{2}+2$.
\item If $\calS$ is $2$-spec and $n \neq 4$ then $\dim \calS \leq \dbinom{n}{2}+3$.
\item If $\calS$ is $2$-spec and $n = 4$ then $\dim \calS \leq \dbinom{n}{2}+4$.
\end{enumerate}
\end{lemma}

\begin{proof}
Let us take a subspace $G$ of codimension $2$ of $V$ such that
$\varphi \otimes x \in \calS$ for all $\varphi \in G^o$ and all $x \in V$ such that $\varphi(x)=0$.
We apply Lemma \ref{lemma:splittinghurdles}.

Assume that $\calS$ is $2$-spec. We consider the mapping $\Phi : u \in \calS \mapsto u_G \in \End(G)$.
By point (d) of Lemma \ref{lemma:splittinghurdles}, every $u \in \Ker (\Phi)$ satisfies $\tr(u^{V/G})=0$, which yields
$\dim(\Ker \Phi) \leq 3+2(n-2)$.

Moreover $\calT=\Phi(\calS)$ is a $1$-spec subspace of $\End(G)$.
Unless $n-2=2$, we know from Theorem \ref{theo:1spec} that $\dim \calT \leq \dbinom{n-2}{2}+1$, and if $n-2=2$ we have $\dim \calT \leq \dbinom{n-2}{2}+2$.
Hence, if $n \neq 4$ then
$$\dim \calS \leq 3+2(n-2)+\dim \calT \leq \dbinom{n}{2}+3,$$
and if $n=4$ then likewise we obtain $\dim \calS \leq \dbinom{n}{2}+4$.

Assume finally that $\calS$ is $1^\star$-spec. Then $\calT$ is $0^\star$-spec and we directly know
from Theorem \ref{theo:0starspec} that $\dim \calT \leq \dbinom{n-2}{2}$, which yields $\dim \calS \leq \dbinom{n}{2}+2$.
\end{proof}

\section{The dimension inequality for fields with large cardinality}\label{section:largefields}

\subsection{The case $n=2$}

Here we solve the case $n=2$, both for the dimension inequality and the structure of optimal spaces, for $1^\star$-spec spaces.

\begin{lemma}\label{lemma:generatesln}
The space $\mathfrak{sl}_n(\F)$ is spanned by its rank $1$ matrices.
\end{lemma}

\begin{proof}
This result actually holds for all linear hyperplanes of $\Mat_n(\F)$ (see e.g., \cite{Azoff}, which connects this property to the classical fact that
any $1$-dimensional operator space is algebraically reflexive) but we give an elementary proof.
It suffices to note that the system consisting of the $E_{i,j}$ matrices with distinct $i,j$ in $\lcro 1,n\rcro$ on the one hand,
and the matrices $E_{1,1}+E_{1,i}-E_{i,1}-E_{i,i}$, with $i \in \lcro 2,n\rcro$ on the other hand, spans the vector space
$\mathfrak{sl}_n(\F)$. Obviously, all these matrices have rank $1$.
\end{proof}

\begin{prop}\label{prop:1starn=2}
Assume that $|\F|>2$.
Let $\calM$ be a $1^\star$-spec subspace of $\Mat_2(\F)$.
Then $\dim \calM \leq 3$, and equality holds only if $\calM=\mathfrak{sl}_2(\F)$.
\end{prop}

\begin{proof}
Consider an arbitrary $2$-dimensional vector space $V$ and a $1^\star$-spec subspace $\calS$ of $\End(V)$.
As there are more than $2$ elements of $\F$, it is clear that $\calS \neq \End(V)$.
Hence $\dim \calS \leq 3$.

Assume now that $\dim \calS=3$.
Let $x \in V \setminus\{0\}$. We wish to prove that $\calS$ contains
a trace zero endomorphism with range $\F x$ (in which case it contains all such endomorphisms).

Assume first that $\hat{x} : u \in \calS \mapsto u(x)\in V$ is non-surjective. Then its kernel has dimension $2$.
Yet this kernel is included in the $2$-dimensional space $\varphi \otimes V$ where $\varphi \in x^\bot \setminus \{0\}$.
Hence these spaces are equal, and we recover $\varphi \otimes x \in \calS$.

Assume next that $\hat{x}$ is surjective.
Take $u_0 \in \calS \setminus \{0\}$ such that $u_0(x)=0$. Hence $u_0$ has rank $1$, to the effect that $\im u_0$ is spanned by an eigenvector of $u_0$. Hence either $x \in \im u_0$, in which case we are done, or $\F x \oplus \im u_0=V$. Assume now that the second case holds, and take $y \in \im u_0 \setminus \{0\}$. Then $u_0(y)=\alpha y$ for some $\alpha \in \F^\times$.
Now, since $\hat{x}$ is surjective we can pick $u_1 \in \calS$ such that $u_1(x)=x$. The respective matrices $M_1$ and $M_0$ of $u_1$ and $u_0$ in $(x,y)$ then have the form
$$M_1=\begin{bmatrix}
1 & ? \\
0 & \beta
\end{bmatrix} \quad \text{and} \quad M_0=\begin{bmatrix}
0 & 0 \\
0 & \alpha
\end{bmatrix} \quad \text{for some $\beta \in \F$.}$$
Since $|\F|>2$, we can pick $\gamma \in \F \setminus \{\alpha^{-1} \beta,\alpha^{-1}(1+\beta)\}$ and note that $M_1+\gamma M_0$ has spectrum
$\{1,\beta+\alpha\gamma\}$, with $\beta+\alpha \gamma \not\in \{0,1\}$.
Therefore $M_1+\gamma M_0$ has two distinct nonzero eigenvalues in $\F$, which is impossible because $u_1+\gamma u_0$ belongs to $\calS$.
Hence in any case we have shown that $x^\bot \otimes x \subseteq \calS$.

By varying $x$, we conclude that $\calS$ contains all the trace zero and rank $1$ endomorphisms of $V$, and by Lemma \ref{lemma:generatesln} it follows that
$\mathfrak{sl}(V) \subseteq \calS$. As the dimensions are equal we conclude that $\calS=\mathfrak{sl}(V)$.
The matrix analogue of the result follows immediately.
\end{proof}

\subsection{Proof of the dimension inequality for $1^\star$-spec spaces}\label{section:endproofinfinite1starspec}

Now, we prove Theorem \ref{theorem:maintheodim1starspec} by induction on $n$ for fields with large cardinality.
The case $n=2$ has just been dealt with in the previous section.
Thus, we take an integer $n \geq 3$ such that $|\F| \geq n$, and we assume that for every $k \in \lcro 2,n-1\rcro$
and every $k$-dimensional vector space $V$ over $\F$, every $1^\star$-spec linear subspace of $\End(V)$
has dimension at most $\dbinom{k}{2}+2$.

Throughout, we let $V$ be an $n$-dimensional vector space over $\F$, and $\calS$ be a $1^\star$-spec linear subspace of $\End(V)$.
We will prove that $\dim \calS \leq \dbinom{n}{2}+2$.

To simplify things, we will perform a \emph{reductio ad absurdum}, by assuming that $\dim \calS \geq \dbinom{n}{2}+3$.
Hence we can extract a linear subspace of $\calS$ with dimension $\dbinom{n}{2}+3$; of course such a subspace remains $1^\star$-spec.
Therefore, in the remainder of the proof we will simply assume that $\dim \calS=\dbinom{n}{2}+3$ and seek to find a contradiction.

Applying point (a) of Lemma \ref{lemma:ineqdimhurdles}, we readily find:

\begin{claim}\label{claim:nohurdle}
The space $\calS$ is not a hurdle.
\end{claim}

Next:

\begin{claim}\label{claim:noadapted}
The space $\calS$ has no adapted vector.
\end{claim}

\begin{proof}
Indeed, by Theorem \ref{theo:0starspec} or the induction hypothesis (whether, for a given $\calS$-adapted vector $x$, the space $\calS$ contains an operator with range $\F x$ or not), and as explained in the beginning of Section \ref{section:strategy}, the existence of such an $\calS$-adapted vector would yield $\dim \calS \leq 2+\dbinom{n}{2}$ in any case, contradicting our assumptions.
\end{proof}

As a consequence, the space $\calS^\bot$ is intransitive.

\begin{claim}\label{claim:primitivelyintransitive}
The space $\calS^\bot$ is primitively intransitive.
\end{claim}

\begin{proof}
Assume the contrary. Then we introduce an intransitivity veil $W$ of $\calS^\bot$ (see the definition on page \pageref{def:intransitivityveil}). 
We set $s:=\dim(V/W)$
and denote by $\pi : V \twoheadrightarrow V/W$ the standard projection.
Note that $1 \leq s < n$.
First of all, we eliminate the possibility that $s=1$.
Indeed, assume for a moment that $s=1$. Then, as $\pi \calS^\bot$ is intransitive it consists only of the zero operator,
which means that all the operators in $\calS^\bot$ map into $W$. Then, by double-orthogonality we gather
that $\calS$ contains all the endomorphisms of $V$ that vanish on $W$.
The First Confinement Lemma then implies that every vector in $V \setminus W$ is $\calS$-adapted, and as there
is at least one such vector we would contradict Claim \ref{claim:noadapted}.

Hence $2 \leq s <n$. Because $\pi \calS^\bot$ is primitively intransitive we deduce from Atkinson's theorem that
$\dim (\pi \calS^\bot) \leq \dbinom{s}{2}$ (note that this is a critical point where we use the assumption $|\F| \geq n$).
Next, we consider the subspace
$$\calS_W:=\{u \in \calS: \forall x \in W, \; u(x)=0\},$$ and for
$u \in \calS_W$ we consider the induced endomorphism $u^{V/W}$ of $V/W$, thereby yielding a $1^\star$-spec linear subspace
$$\calS^{V/W}:=\bigl\{u^{V/W} \mid u \in \calS_W\bigr\}$$
of $\End(V/W)$.
By induction we have
$$\dim(\calS^{V/W}) \leq \dbinom{s}{2}+2,$$
and obviously
$$\dim \calS_W \leq s(n-s)+\dim(\calS^{V/W}).$$
Besides, the Second Trace Orthogonality Lemma yields
$$\dim \calS_W=ns-\dim (\pi \calS^\bot) \geq ns-\dbinom{s}{2}.$$
Noting that
$$s(n-s)+\dbinom{s}{2}+2-ns+\dbinom{s}{2}=2-s,$$
we deduce that $s=2$, that $\dim (\calS^{V/W})=3$ and that
$\calS_W$ contains all the endomorphisms of $V$ that vanish on $W$ and map into $W$.
By Proposition \ref{prop:1starn=2}, it follows from the second point that $\calS^{V/W}=\mathfrak{sl}(V/W)$, and we combine this with the third point to find that $\calS$
is a hurdle. This contradicts Claim \ref{claim:nohurdle}.

Therefore $\calS^\bot$ is primitively intransitive.
\end{proof}

\begin{Rem}\label{remark:primintransitive}
For future reference, it is important to stress that the proof of Claim \ref{claim:primitivelyintransitive} does not directly use any assumption on the dimension of $\calS$: it works as long as we know that $\calS$ is $2$-spec, that $\calS^\bot$ is intransitive, that $\dim \calS^{V/W} \leq \dbinom{s}{2}+2$, and
that no linear hyperplane of $V$ contains all the non-$\calS$-adapted vectors (note that in theory the space $\calS^\bot$ can be intransitive without the
absence of $\calS$-adapted vectors: an example is the space of all symmetric matrices over the real numbers).
\end{Rem}

We are ready to obtain a final contradiction.
We can now apply Atkinson's theorem, remembering the crucial assumption that $|\F| \geq n$.
By point (a) in Atkinson's theorem, we find $\dim \calS^\bot \leq \dbinom{n}{2}$.
However, our assumptions show that
$$\dim \calS^\bot =\dbinom{n+1}{2}-3=\dbinom{n}{2}+n-3.$$
Hence $n=3$ and $\dim \calS^\bot=\dbinom{n}{2}$.
Next, point (b) in Atkinson's theorem yields a non-degenerate bilinear form $b : V^2 \rightarrow \F$
that is an alternator of $\calS^\bot$, i.e.,
$$\forall x \in V, \; \forall u \in \calS^\bot, \; b(x,u(x))=0.$$
Let us now represent $\calS$ by a matrix space $\calM$ in an arbitrary basis.
Denote by $P$ the Gram matrix of that basis for $b$ (that is, if the basis is $(e_1,\dots,e_n)$ one has $P=(b(e_i,e_j))_{1 \leq i,j \leq n}$).
The fact that $b$ is an alternator of $\calS^\bot$
then translates into $\calM^\bot \subseteq P^{-1} \Mata_n(\F)$, and since the dimensions are equal we
obtain $\calM^\bot = P^{-1} \Mata_n(\F)$. In turn, this yields
$$\calM=\Mats_n(\F) P.$$
Now, we do not forget the assumption that, for each $x \in \calS$, there is an operator with range $\F x$ and trace $0$ in $\calS$
(i.e., there is no $\calS$-adapted vector). Yet, for all $X \in \F^n$, there is up to scalar multiplication exactly one matrix $S$
of $\Mats_n(\F)$ with range $\F X$, and this has the consequence that $\tr(SP)=0$.
Since $\Mats_n(\F)$ is spanned by its rank $1$ matrices, we deduce that $P \in \Mata_n(\F)$, which is absurd because $n$ is odd and $P$ is invertible!

Hence, thanks to this final contradiction the inductive step is now climbed.

\subsection{Proof of the dimension inequality for $2$-spec spaces}\label{section:endproofinfinite2spec}

The proof pattern here is largely similar to the one of the previous section,
the difference being that we no longer rely on induction, rather we directly use Theorem \ref{theorem:maintheodim1starspec}.

So, we let $V$ be a finite-dimensional vector space with dimension $n \geq 3$, and
$\calS$ be a $2$-spec linear subspace of $\End(V)$.
We assume that $|\F| \geq n$.
We will consider three cases separately, whether $n=3$ or $n=4$ or $n\geq 5$.

\vskip 3mm
\noindent \textbf{Case 1: $n \geq 5$.}
It suffices to prove that the additional assumption $\dim \calS=\dbinom{n}{2}+4$
leads to a contradiction. Following the chain of arguments from the preceding section, we successively obtain:
\begin{itemize}
\item That $\calS$ has no adapted vector: otherwise we take an $\calS$-adapted vector $x$
and obtain $\dim \calS \leq \dim(\overline{\calS_{x,0}})+n$ with the notation from Section \ref{section:strategy},
and $\dim(\overline{\calS_{x,0}}) \leq \dbinom{n-1}{2}+2$ by Theorem \ref{theorem:maintheodim1starspec}, which leads to
$\dim \calS \leq \dbinom{n}{2}+3$.
\item That $\calS$ is not a hurdle, otherwise we would have $\dim \calS \leq \dbinom{n}{2}+3$ by Lemma \ref{lemma:ineqdimhurdles}.
\end{itemize}
Finally, we can follow the proof of Claim \ref{claim:primitivelyintransitive} to find that $\calS^\bot$ is primitively intransitive:
we replace the use of the induction hypothesis for the space $\calS^{V/W}$ with a direct call to Theorem \ref{theorem:maintheodim1starspec},
since $\calS^{V/W}$ is actually $1^\star$-spec. The details are essentially similar.

Finally, we can apply Atkinson's theorem and obtain $\dim \calS^\bot \leq \dbinom{n}{2}$.
Yet $\dim \calS^\bot=\dbinom{n+1}{2}-4=\dbinom{n}{2}+n-4>\dbinom{n}{2}$.
This final contradiction completes the proof when $n \geq 5$.

\vskip 3mm
\noindent \textbf{Case 2: $n=3$.} \\
The trick here is to consider the subspace
$$\calS_0:=\{u \in \calS : \tr u=0\}$$
and to note that it is $1^\star$-spec.
Indeed, let $u \in \calS_0$. Assume that $u$ has two distinct nonzero eigenvalues $\alpha,\beta$ in $\F$.
Then it is triangularizable, and its third eigenvalue $\gamma$ satisfies $\gamma=\tr u-\alpha-\beta=\alpha+\beta$.
Then $\gamma \neq \alpha$ and $\gamma \neq \beta$, which contradicts the assumption that $\calS$ is $2$-spec.
It follows that
$$\dim \calS \leq 1+\dim \calS_0 \leq 1+\dbinom{n}{2}+2=\binom{n}{2}+3.$$

\vskip 3mm
\noindent \textbf{Case 3: $n=4$.} \\
This time around, we assume that $\dim \calS =\dbinom{n}{2}+5$ and seek to find a contradiction.
Now $\dim \calS^\bot=\dbinom{n+1}{2}-5=\dbinom{n}{2}-1 > \dbinom{n-1}{2}$.

Then we note that $\calS$ has no \emph{weakly} adapted vector. Indeed, if $x$ is a weakly $\calS$-adapted vector then
$$\dim \calS \leq n+1+\dim \overline{\calS_{x,0}} \leq n+1+\dbinom{n-1}{2}+3 \leq \binom{n}{2}+4.$$
It follows from Equation \eqref{eq:transrank} that the transitive rank in $\calS^\bot$ is less than $n-1$.
By point (a) of Atkinson's theorem, the space $\calS^\bot$ cannot be primitively intransitive.
However, we can follow the steps of the proof in Section \ref{section:endproofinfinite1starspec}, and we successively find:
\begin{itemize}
\item That $\calS$ is not a hurdle, otherwise we would deduce from Lemma \ref{lemma:ineqdimhurdles}
that $\dim \calS \leq \dbinom{n}{2}+4$;
\item That $\calS$ has no adapted vector (because it has no weakly adapted vector);
\item That $\calS^\bot$ is primitively intransitive (see Remark \ref{remark:primintransitive}).
\end{itemize}
This is contradictory, and the proof of Theorem \ref{theorem:maintheodim2spec} is now complete (still assuming that $|\F| \geq n$, of course).

\subsection{The optimal $2$-spec spaces for $n=4$}

Here we seize the opportunity to give a quick proof of the classification theorem
for the optimal $2$-spec subspaces of $\Mat_4(\F)$, thanks to an easy adaptation of the previous line of reasoning.
So, let $V$ be a $4$-dimensional vector space, and $\calS$ be a $2$-spec linear subspace of $\End(V)$ with dimension $\binom{4}{2}+4$.
We distinguish two cases, whether $\calS$ is a hurdle or not.
In any case, note that $\F \id_V+\calS$ is still $2$-spec, and hence by the optimality of $\calS$ we gather
that $\id_V \in \calS$.

\vskip 3mm
\noindent
\textbf{Case 1: $\calS$ is a hurdle.} \\
We find a matrix space $\calM$ that includes $\{0_2\} \vee \mathfrak{sl}_2(\F)$
and represents $\calS$ in some basis of $V$.
By Lemma \ref{lemma:splittinghurdles}, any matrix $M \in \calM$ has the form
$$M=\begin{bmatrix}
K(M) & [?]_{2 \times 2} \\
[0]_{2 \times 2} & s(M)
\end{bmatrix},$$
where $K(\calM)$ is a $1$-spec subspace of $\Mat_2(\F)$.
Moreover, the kernel of $M \mapsto K(M)$ is precisely $\{0_2\} \vee \mathfrak{sl}_2(\F)$, by point (c) of Lemma
\ref{lemma:splittinghurdles}. By the rank theorem
$$\dim \calS= \dim K(\calS)+3+4$$
and hence $\dim K(\calS)=3$. Hence $K(\calS)=\mathfrak{sl}_2(\F)$, by Theorem \ref{theo:1spec}.
In particular, for every $M$ such that $K(M)$ has rank $1$ (and hence at least one eigenvalue in $\F$), we deduce from point
(c) of Lemma \ref{lemma:splittinghurdles} that $\tr(s(M))=0$.
Since $K(\calM)$ is spanned by its rank $1$ matrices we conclude that $s(\calM) \subseteq \mathfrak{sl}_2(\F)$.
It follows that $\calM \subseteq \mathfrak{sl}_2(\F) \vee \mathfrak{sl}_2(\F)$, and since the dimensions are equal we
conclude that $\calM =\mathfrak{sl}_2(\F) \vee \mathfrak{sl}_2(\F)$.

\vskip 3mm
\noindent
\textbf{Case 2: $\calS$ is not a hurdle.}

In that case we use exactly the same line of reasoning as for the case $n \geq 5$ in the previous section. Indeed, we successively find:
\begin{itemize}
\item That $\calS$ has no adapted vector, otherwise $\dim \calS \leq \dbinom{n}{2}+3$;
\item As a consequence, that $\calS^\bot$ is intransitive;
\item We can then follow once more the proof of Claim \ref{claim:primitivelyintransitive} to see that $\calS^\bot$ is primitively intransitive
(see Remark \ref{remark:primintransitive}).
\end{itemize}
Here we have $\dim \calS^\bot=6=\dbinom{n}{2}$ and $|\F| \geq 4$, so Atkinson's theorem can be applied.
It follows that $\calS^\bot$ has a non-degenerate alternator $b$.
Now, take a basis $\bfB$ of $V$, and denote by $P$ the Gram matrix of that basis for $b$.
Like in Section \ref{section:endproofinfinite1starspec} we obtain that $\calS$ is represented by the matrix space $\Mats_4(\F) P$.
And likewise, we deduce from the lack of $\calS$-adapted vectors that $\tr(SP)=0$ for all $S \in \Mats_4(\F)$.
Hence $P$ is alternating, to the effect that $b$ is a symplectic form. Then we write $\Mats_4(\F)P=(P^T)^{-1}(P^T \Mats_4(\F)P)=(P^T)^{-1} \Mats_4(\F)=P^{-1} \Mats_4(\F)$,
and we conclude that $\calS$ is the space of all $b$-symmetric endomorphisms of $V$.
This yields the second stated conclusion in Theorem \ref{theo:dim4}.

Hence the proof of Theorem \ref{theo:dim4} is now complete.

\section{Preparatory work for the general case}\label{section:preparationforlastpart}

Here, $\F$ denotes an arbitrary field (with no restriction of characteristic nor of cardinality).

\subsection{Complexes of linear subspaces}

\begin{Def}
Let $k \geq 2$ and $V$ be an $n$-dimensional vector space.
A \textbf{$k$-complex} of $V$ is a $(k-1)n$-list $(V_1,\dots,V_{(k-1)n})$ of linear subspaces
of $V$ such that $\dim V_i=1+\lfloor \frac{i-1}{k}\rfloor$ for all $i \in \lcro 1,(k-1)n\rcro$,
i.e., the first $k$ spaces have dimension $1$, the next $k$ spaces have dimension $2$, and so on.
\end{Def}

In \cite{dSPfeweigenvalues}, we used $2$-complexes only. Here, we will need to use $3$-complexes as well.
We recall the following lemma from \cite{dSPfeweigenvalues} (lemma 2.5 there):

\begin{lemma}[Covering Lemma]\label{coveringlemma2}
Let $r$ be a positive integer such that $|\F|>r$.
Let $V$ be an $n$-dimensional vector space over $\F$, and $(V_i)_{i \in I}$
be a family of linear subspaces of $V$ in which:
\begin{enumerate}[(i)]
\item $|I|=(n-1)r+1$ ;
\item For all $k \in \lcro 1,n-2\rcro$, exactly $r$ vector spaces in the family have dimension $k$;
\item Exactly $r+1$ vector spaces in the family have dimension $n-1$.
\end{enumerate}
Then the subspaces $V_i$ do not cover $V$.
\end{lemma}

\subsection{Complexes vs polynomials}

The next lemma can be seen as a generalization of the Covering Lemma. It will be very useful in the remainder of the article.
Because we are dealing with fields that may be finite, it is useful to state the following definition precisely:

\begin{Def}
Let $d$ be a positive integer and $V$ be a vector space over $\F$.
A \textbf{$d$-homogeneous polynomial function} is an element of the linear span of
the functions of the form $\prod_{k=1}^d \varphi_k$, where $\varphi_1,\dots,\varphi_d$ are linear forms on $V$,
in the vector space $\F^V$ of all functions from $V$ to $\F$.
\end{Def}

Denote by $S(V)$ the \emph{dual} symmetric algebra on $V$, and by $S_d(V)$ its subspace of all $d$-homogenous elements (for an integer $d \geq 0$).
In particular $S_1(V)$ is the dual vector space of $V$.
Classically, if $|\F| \geq d$ then specializing yields a bijection from $S_d(V)$ to the set of all $d$-homogeneous polynomial functions on
$V$.

\begin{lemma}[Vanishing Lemma for Homogeneous Polynomials]\label{lemma:vanishinghomogeneous}
Let $d$ be a positive integer such that $|\F| \geq d$.
Let $V$ be a vector space with finite dimension $n$ over $\F$, and
$p : V \rightarrow \F$ be a $d$-homogeneous polynomial function.

Assume that we have a finite family $(V_i)_{i \in I}$ of nontrivial linear subspaces of $V$ in which:
\begin{enumerate}[(i)]
\item For each $j \in \lcro 1,n-2\rcro$, there are at most $|\F|-1$ indices $i$ such that
$\dim V_i=j$;
\item There are at most $|\F|-d$ indices $i$ such that $\dim V_i=n-1$;
\item The function $p$ vanishes outside of $\underset{i \in I}{\bigcup} V_i$.
\end{enumerate}
Then $p=0$.
\end{lemma}

\begin{proof}
The proof is by induction on $n$.
The case $n=1$ is trivial, as the family $(V_i)$ is void in that case, to the
effect that assumption (iii) requires that $p=0$.

Assume now that $n \geq 2$. Then we consider the dual family $(V_i^\ortho)_i$ in the
dual space $V^\star$. Note that this family contains at most
$|\F|-1$ subspaces of dimension $i$ for all $i \in \lcro 2,n-1\rcro$, and at most
$|\F|-d$ subspaces of dimension $1$.

Then, by induction we can use the Covering Lemma to obtain $d+1$ \emph{pairwise} linearly independent elements in $V^\star$, all of them outside of
$\underset{i \in I}{\bigcup} V_i^\ortho$.
Indeed, suppose that we have such elements $f_1,\dots,f_i$ for some $i \in \lcro 0,d\rcro$.
Then by adjoining the lines $\F f_j$ to the spaces $V_i^{\circ}$, we can reduce the situation to the one of the Covering Lemma 
applied to $r:=|\F|-1$ (by adding some extra subspaces if necessary):
if $i<d$ this is obvious, whereas if $i=d$ we view $\F f_1,\dots,\F f_{d-1}$ as $1$-dimensional spaces but we single out
$\F f_d$ and embed it inside a subspace of dimension $n-1$ (and hence we can take advantage of assumption (iii) in the Covering Lemma).

Now, take linear forms $f_1,\dots,f_{d+1}$ in $V^\star$ that satisfy the previous condition.
Take $f_1$ for instance. Then by assumption none of the $V_j$'s is included in $\Ker f_1$,
so by taking $J:=\{i \in I : \dim V_j \neq 1\}$, we find that
$(V_j \cap \Ker f_1)_{j \in J}$ satisfies assumptions (i) and (ii) with respect to the space $\Ker f_1$, whereas
$V_j \cap \Ker f_1=\{0\}$ for all $j \in I \setminus J$.
By induction, $p$ vanishes on $\Ker f_1$. Likewise $p$ vanishes on $\Ker f_i$ for all $i \in \lcro 1,d+1\rcro$.

We are ready to conclude that $p=0$. Indeed, by taking the homogeneous elements
$\mathbf{p},\mathbf{f_1},\dots,\mathbf{f_{d+1}}$ of the symmetric tensor space $S(V)$ that correspond to $p,f_1,\dots,f_{d+1}$, respectively,
with respective degrees $d,1,\dots,1$, we deduce from the linear Nullstellensatz that $\mathbf{f_i}$ divides $\mathbf{p}$ (because here $|\F| \geq d$)
for all $i \in \lcro 1,p\rcro$. As $\mathbf{f_1},\dots,\mathbf{f_{d+1}}$ are pairwise linearly independent and are irreducible, and hence pairwise relatively prime,
$\mathbf{f_1} \cdots \mathbf{f_{d+1}}$ divides $\mathbf{p}$ because $S(V)$ is a unique factorization domain.
Finally, we must have $\mathbf{p}=0$ because $\mathbf{p}$ has homogeneous degree $d$.
Hence $p=0$ and the inductive step is climbed, which completes the proof.
\end{proof}

We finish by noting that the Covering Lemma can now be viewed as a special case of the Vanishing Lemma for Homogeneous Polynomials (with $d=1$).
Indeed, with the assumptions of the Covering Lemma, we pick one of the hyperplanes $H$ in the family $(V_i)_{i \in I}$
and write it as the kernel of some nonzero linear form $f$ on $V$.
If the union of $H$ and the remaining subspaces were equal to $V$ then
$f$ would vanish  outside of the union of the remaining subspaces of the family,
and Lemma \ref{lemma:vanishinghomogeneous} would yield $f=0$, a contradiction.

\section{The general case}\label{section:anyfield}

In this section, we prove Theorems \ref{theorem:maintheodim1starspec} and \ref{theorem:maintheodim2spec}
in their full generality, thereby removing all the assumptions on the cardinality of $\F$ except that $|\F|>2$.

The proof will actually yield the following more precise results, which will be useful to analyse optimal spaces in a future article:

\begin{theo}\label{theorem:1starspecwithadapted}
Let $\F$ be a field with $\car(\F)=2$ and $|\F| > 2$, and let $n \geq 3$. Let $\calS$ be a $1^\star$-spec linear subspace of endomorphisms of
an $n$-dimensional vector space. Then $\dim \calS \leq \dbinom{n}{2}+2$, and either $\calS$ is a hurdle or $\calS$ has an adapted vector.
\end{theo}

\begin{theo}\label{theorem:2specwithadapted}
Let $\F$ be a field with $\car(\F)=2$ and $|\F| > 2$, and let $n$ be an integer such that $n=3$ or $n \geq 5$. Let $\calS$ be a $2$-spec linear
subspace of endomorphisms of an $n$-dimensional vector space. Then $\dim \calS \leq \dbinom{n}{2}+3$, and if $n \geq 5$ then either $\calS$ is a hurdle or $\calS$ has an adapted vector.
\end{theo}

The main difficulty is the proof of Theorem \ref{theorem:1starspecwithadapted}, as the proof of Theorem \ref{theorem:2specwithadapted} will follow a similar pattern.
As we have explained in the introduction, the innovation is that here
everything will be proved almost at once, as we will not separate the dimension inequality from the work on adapted vectors in the inductive proof.

From now on, $\F$ denotes a field with characteristic $2$ and more than $2$ elements.

\subsection{The induction hypothesis}

Here the induction hypothesis is quite large, and now we state the precise property as follows:

\begin{prop}\label{prop:inductionprop}
Let $n \geq 2$, and $\calS$ be a $1^\star$-spec linear subspace of endomorphisms of
an $n$-dimensional vector space $V$. Then:
\begin{enumerate}[(i)]
\item $\dim \calS \leq \dbinom{n}{2}+2$;
\item If $n \geq 3$ then the weakly $\calS$-adapted vectors are not all contained in the union of some $3$-complex of $V$.
\item If $n \geq 3$ then either $\calS$ is a hurdle or the $\calS$-adapted vectors are not all contained in the union of some $2$-complex of $V$.
\end{enumerate}
\end{prop}

Note that the requirement $n \geq 3$ in point (ii) is unavoidable, since every $3$-complex of a $2$-dimensional vector space $V$ contains $V$ as its fourth component.

In the remainder of this section, we assume that $n \geq 3$.
In case $n \leq 4$, we will take advantage of the fact that the dimension inequality has already been proved in Section \ref{section:largefields},
since $|\F| \geq 4$.

From now on, we take an $n$-dimensional vector space $V$ and a $1^\star$-spec subspace $\calS$ of $\End(V)$, and we assume that the result
of Proposition \ref{prop:inductionprop} holds over every $\F$-dimensional vector space with dimension $k \in \lcro 2,n-1\rcro$.

\subsection{A first rough upper-bound on the dimension}

Now, we will find a rough estimate for the dimension, which will be useful later on.

\begin{lemma}\label{lemma:diagonalzero}
Let a linear subspace $\calM$ of $\Mat_n(\F)$ contain all the matrices with diagonal zero, with $n \geq 3$. Then
$\calM$ is not $2$-spec.
\end{lemma}

\begin{proof}
Indeed, we note that $\calM$ contains all the companion matrices of polynomials of $\F[t]$
that are monic, with degree $n$ and trace zero. Yet since $n \geq 3$ the polynomial $t^{n-3}(t-\alpha)(t-\beta)(t-(\alpha+\beta))$ is such a polynomial, where we choose
$\alpha$ and $\beta$ as distinct nonzero elements of $\F$.
\end{proof}

\begin{claim}\label{claim:roughbound}
One has $\dim \calS \leq n+1+\dbinom{n}{2}\cdot$
\end{claim}

\begin{proof}
Choose an arbitrary vector $x \in V \setminus \{0\}$.
Assume first that $\calS$ contains an element of $V^\star \otimes x$ with nonzero trace.
We can use the line of reasoning outlined in Section \ref{section:strategy} to find
a $0^\star$-spec subspace $\calT$ of $\End(V/\F x)$ such that $\dim \calS \leq n+(n-1)+\dim \calT$, and then we use
Theorem \ref{theo:0starspec} to deduce that $\dim \calS \leq n+(n-1)+\dbinom{n-1}{2}=n+\dbinom{n}{2}$.

Assume now that all the elements in $\calS \cap (V^\star \otimes x)$ have trace zero. Then we find a $1^\star$-spec subspace $\calT$ of $\End(V/\F x)$ such that
$$\dim \calS \leq \dim(x^\bot \otimes x)+(n-1)+\dim \calT \leq 2(n-1)+\dim \calT,$$
and by induction $\dim \calT \leq \dbinom{n-1}{2}+2$. This yields the claimed inequality.
\end{proof}

\subsection{Proof of point (ii) for $n=3$}

Here, we assume $n=3$ and give a direct proof of point (ii).
So, assume that we have a $3$-complex $(V_i)_{1 \leq i \leq 6}$ of $V$ whose union contains all the weakly $\calS$-adapted vectors.
By the Covering Lemma, the subset $E \setminus \underset{i=1}{\overset{6}{\bigcup}} V_i$ is not included in a linear hyperplane of 
$V$ (otherwise, we take such a hyperplane $V_7$ and note that $(V_i)_{1 \leq i \leq 7}$ covers $V$, thereby contradicting the Covering Lemma for $r=3$).
Hence, we can take a basis $(e_1,e_2,e_3)$ of $V$ that consists of vectors that are not weakly $\calS$-adapted. Since $n=3$
it follows that $e_i^\bot \otimes e_i \subseteq \calS$ for all $i \in \{1,2,3\}$.
Hence, the matrix space $\calM$ that represents $\calS$ in the basis $(e_1,e_2,e_3)$ contains all the $3$-by-$3$ matrices with diagonal zero, 
and this contradicts Lemma \ref{lemma:diagonalzero}.

Hence point (ii) holds for $n=3$.

\subsection{Completing the proof of point (ii), and an improved upper-bound on the dimension}\label{section:inductiveweaklyadapted}

We now prove point (ii), assuming $n \geq 4$.
The method is an adaptation of the one featured in \cite{dSPfeweigenvalues}.
We use a \emph{reductio ad absurdum}, by assuming that there is a $3$-complex
$(V_1,\dots,V_{2n})$ of $V$ whose union contains all the weakly $\calS$-adapted vectors.

Since $|\F| \geq 4$, the Covering Lemma (applied to $p=3$) shows that there is no linear hyperplane $H$ of $V^\star$ such that $H \cup \underset{i=1}{\overset{2n}{\bigcup}} V_i^o=V^\star$,
which yields a basis $(f_1,\dots,f_n)$ of $V^\star$ made of vectors that are all outside of $\underset{i=1}{\overset{2n}{\bigcup}} V_i^o$.
In other words, no $f_i$ vanishes on some $V_k$.

Now, we fix an $i \in \lcro 1,n\rcro$ and consider the linear hyperplane $H:=\Ker f_i$.
Let us consider the space $\calS^H=\calS \cap \Hom(U,H)$ and the induced subspace
$$\calT:=\{u_H \mid u \in \calS^H\} \subseteq \End(H).$$
Then $\calT$ is a $1^\star$-spec linear subspace of $\End(H)$.

Moreover the intersection $\calS^H \cap (f_i \otimes V)$ equals
$f_i \otimes W_i$ for some linear subspace $W_i$ of $H$.

We shall prove that
$$\dim W_i> \left\lfloor \frac{2n}{3}\right\rfloor.$$
Assume on the contrary that $\dim W_i \leq \lfloor \frac{2n}{3}\rfloor=1+\lfloor \frac{2n-3}{3}\rfloor$,
and let us embed $W_i$ inside a linear subspace $G$ of $H$ with dimension $1+\lfloor \frac{2n-3}{3}\rfloor$
(indeed, this is possible since clearly $n> \lfloor \frac{2n}{3}\rfloor$).
Since none of the $V_j$'s is included in the kernel of $f_i$, we have
$\dim (V_j \cap H)=\dim(V_j)-1$ for all $j \in \lcro 1,2n\rcro$, and it follows that
$(V_4 \cap H,\dots,V_{2n} \cap H,G)$ is a $3$-complex of $H$.

By induction applied to $\calS^H$, there exists a vector $x \in H$ that is weakly $\calT$-adapted and belongs to none of
$V_4 \cap H,\dots,V_{2n} \cap H,G$. Hence $x$ does not belong to $V_1 \cup \cdots \cup V_{2n}$
(remember that $V_i \cap H=\{0\}$ for all $i \in \{1,2,3\}$), and therefore
$x$ is not weakly $\calS$-adapted. This yields a $2$-dimensional linear subspace $P$ of $x^\bot$ such that $P \otimes x \subseteq \calS$.
Then $\calT$ contains all the operators $\varphi_{|H} \otimes x$ with $\varphi \in P$, and the space of all these operators
has dimension at most $1$ because $x$ is weakly $\calT$-adapted. Hence the linear mapping $\varphi \in P \mapsto \varphi_{|H} \otimes x$,
is noninjective, which yields $\varphi \in P \setminus \{0\}$ such that $\varphi_{|H}=0$. Hence $\varphi=\lambda f_i$ for some $\lambda \in \F^\times$,
and we deduce that $x \in W_i$. This contradicts the fact that $x \not\in G$.

We deduce from the previous step that
$$\dim W_i \geq 1+\left\lfloor \frac{2n}{3}\right\rfloor.$$
Now, since $f_1,\dots,f_n$ are linearly independent the subspaces $f_i \otimes W_i$ are linearly independent in $\End(V)$, and we deduce that
$$\dim \calS \geq \sum_{i=1}^n \dim(f_i \otimes W_i) \geq n+n\,\left\lfloor \frac{2n}{3}\right\rfloor.$$
As we will now see, this contradicts our previous upper bound on $\dim \calS$.
Indeed, the rough upper bound from Claim \ref{claim:roughbound} was $\dbinom{n}{2}+n+1$, whereas
$$n+n\left\lfloor \frac{2n}{3}\right\rfloor-\dbinom{n}{2}-n-1 \geq \frac{n(2n-2)}{3}-\frac{n(n-1)}{2}-1=\frac{n(n-1)}{6}-1>0$$
since $n \geq 4$. This contradiction completes the proof of point (ii).
In turn, by using a weakly $\calS$-adapted vector we can now use
one more round of the basic inductive technique to obtain the improved upper-bound:

\begin{claim}\label{claim:sharperbound}
One has $\dim \calS \leq \dbinom{n}{2}+3$.
\end{claim}

\subsection{Proof of points (i) and (iii) in case $n \in \{3,4\}$}

Assume that $n \in \{3,4\}$. Then we readily know from Section \ref{section:largefields} that the dimension inequality
$\dim \calS \leq \binom{n}{2}+2$ holds true.

Next, we prove point (iii) directly (i.e., without resorting to an inductive argument).
So, we assume that there is a $2$-complex $(V_1,\dots,V_n)$ of $V$ whose union contains all the
$\calS$-adapted vectors.
Note that $\dim \calS^\bot=n^2-\dim \calS \geq \dbinom{n+1}{2}-2 =\dbinom{n}{2}+(n-2)$ and in particular $\dim \calS^\bot \geq n$.
Since $n \geq 3$, this readily forbids $\calS^\bot$ to be primitively intransitive, because $|\F| \geq 4 \geq n$.

Set $m:=\dim \calS^\bot$.
Now, let us consider a basis $(v_1,\dots,v_m)$ of $\calS^{\bot}$ and a basis $(e_1,\dots,e_n)$ of $V$.
For $x \in V$, we write the matrix of the linear operator $v \in \calS^\bot \mapsto v(x) \in V$ in these bases as $M(x)$.
Finally, let us consider a subset $I \subseteq \lcro 1,m\rcro$ with cardinality $n$, and denote by
$\Delta_I(x)$ the $n \times n$ minor determinant of $M(x)$ obtained by selecting the column indices in $I$.
Note that $\Delta_I$ is an $n$-homogeneous polynomial function on $V$, and that $\Delta_I$ vanishes at every vector of $V$ that is not $\calS$-adapted.
As a consequence, $\Delta_I$ vanishes on the complement of $V_1 \cup \cdots \cup V_n$ in $V$.

By Lemma \ref{lemma:vanishinghomogeneous} applied to $d:=n$, the function $\Delta_I$ actually vanishes.
Indeed here $|\F|-1 \geq 2$, so condition (i) in that lemma is satisfied; and
finally condition (ii) is satisfied because either $n=3$ and there is exactly one hyperplane in the family $(V_1,\dots,V_n)$,
or $n=4$ and there is none.

Varying $I$ yields that $\calS^\bot$ is intransitive. As we have already shown however that $\calS^\bot$
is not primitively intransitive, we can take an intransitivity weil $W$ of it,
and set $s:=\dim (V/W) \in \lcro 1,n-1\rcro$.
This leaves us with three cases to consider.
\begin{itemize}
\item If $s=1$ then we apply the line of reasoning from Claim \ref{claim:primitivelyintransitive},
and we derive that $\calS$ satisfies the conditions of the First Confinement Lemma. This would yield a linear hyperplane $H$ of $V$ that contains all the
vectors that are not $\calS$-adapted, and then $V_1,\dots,V_n,H$ would cover $V$, which obviously contradicts the Covering Lemma.
\item If $s=2$ then we apply the line of reasoning from the proof of Claim \ref{claim:primitivelyintransitive} to find that
$\calS$ is a hurdle (note that this does not really require the induction hypothesis, rather it involves the fact that every $1^\star$-spec subspace of
$\End(V/W)$ has dimension at most $3$, with equality holding only for $\mathfrak{sl}(V/W)$).
\item If $s=3$ and $n=4$ then $\dim \calS^\bot \leq \dim \pi \calS^\bot+\dim \Hom(V,W)$ for the standard projection $\pi : V \twoheadrightarrow V/W$,
and by Atkinson's theorem $\dim(\pi \calS^\bot) \leq \dbinom{n-1}{2}$. Hence in that case
$\dim \calS^\bot \leq n+\dbinom{n-1}{2}= \dbinom{n}{2}+1$, thereby contradicting the fact that $\dim \calS^\bot \geq \dbinom{n}{2}+n-2$ and $n-2 \geq 2$.
\end{itemize}

This completes the proof of point (iii) for $n \leq 4$.

\subsection{Setting up the proof of point (iii) for $n \geq 5$}\label{section:startinductiveadapted}

From now on we systematically assume that $n \geq 5$.

It remains to prove points (i) and (iii). Once point (iii) is obtained, point (i)
follows from it: indeed, by combining point (iii) with the Covering Lemma, we obtain that either
$\calS$ has an adapted vector, in which case we proceed as before to obtain the inequality $\dim \calS \leq \dbinom{n}{2}+2$,
or $\calS$ is a hurdle, in which case this inequality is immediately known by Lemma \ref{lemma:ineqdimhurdles}.
Note that this will also give the second conclusion in Theorem \ref{theorem:1starspecwithadapted}.

Hence, we now only care about proving point (iii), which will turn out to be very technical.
The proof strategy is very similar to the one used in Section \ref{section:inductiveweaklyadapted}
but extra difficulties are created by the possibility of stumbling upon hurdles when trying to apply the induction hypothesis.

We assume throughout that there is a $2$-complex $(V_1,\dots,V_n)$ of $V$ whose union contains all the $\calS$-adapted vectors,
and we aim at proving that $\calS$ is a hurdle.

Just like in Section \ref{section:inductiveweaklyadapted}, we use the Covering Lemma in the dual vector space $V^\star$
to find a basis $(f_1,\dots,f_n)$ of $V^\star$ such that $V_j \not\subseteq \Ker f_i$ for all $i,j$ in  $\lcro 1,n\rcro$.
Again, for $i \in \lcro 1,n\rcro$ we set
$$W_i:=\{x \in \Ker f_i : f_i \otimes x \in \calS\}.$$
Assume temporarily that $\dim W_i>\lfloor \frac{n}{2}\rfloor$ for \emph{all} $i \in \lcro 1,n\rcro$.

As in Section \ref{section:inductiveweaklyadapted}, we deduce that
$$\dim \calS \geq n+n \left\lfloor \frac{n}{2}\right\rfloor.$$
However
$$n+n \left\lfloor \frac{n}{2}\right\rfloor-3-\dbinom{n}{2} \geq n+n \frac{n-1}{2}-3-\frac{n(n-1)}{2}=n-3>0,$$
thereby contradicting Claim \ref{claim:sharperbound}.

We deduce that there exists an index $i \in \lcro 1,n\rcro$ such that $\dim W_i \leq  \lfloor \frac{n}{2}\rfloor$, and without loss of generality we will assume that
$\dim W_1 \leq  \lfloor \frac{n}{2}\rfloor$. From now on, we will completely forget about the other $W_j$ spaces, focusing instead on $W_1$.
Now, as in Section \ref{section:inductiveweaklyadapted} we introduce the space $H:=\Ker f_1$, the subspace $\calS^H=\calS \cap \Hom(V,H)$
and the induced subspace
$$\calT:=\{u_H \mid u \in \calS^H\}.$$
Let us take a linear subspace $G$ of $H$ that includes $W_1$ and has dimension $\lfloor\frac{n}{2}\rfloor$.
This time around, we note that $\calV'=(V_3 \cap H,\dots,V_n \cap H,G)$ is a $2$-complex of $H$.
Now, assume that there exists a $\calT$-adapted vector $x$ outside of the union of $\calV'$.
In particular, $x$ belongs to none of $V_1,\dots,V_n$, so by the starting assumption it is not $\calS$-adapted,
yielding $\varphi \in x^\bot \setminus \{0\}$ such that $\varphi \otimes x \in \calS$.
Then $\varphi_{|H}(x)=\varphi(x)=0$, and $\varphi_{|H} \otimes x \in \calT$.
Since $x$ is $\calT$-adapted this yields $\varphi_{|H}=0$, to the effect that $\varphi \in \F f_1$.
In turn, this shows that $f_1 \otimes x \in \calS$, to the effect that $x \in W_1$. This contradicts our assumption that $x \not\in G$.
Hence, the union of the $2$-complex $\calV'$ contains all the $\calT$-adapted vectors.

Now we use the induction hypothesis, and recover that $\calT$ is a hurdle.
This is where things become much more difficult than before.

\subsection{The proof of point (iii), continued}

We have just obtained that $\calT$ is a hurdle and $\dim W_1 \leq \lfloor n/2\rfloor$, and now we will analyze these new facts in depth.
We take a $2$-dimensional linear subspace $P$ of $H^\star$ that is associated with the hurdle $\calT$, and its pre-orthogonal
$$G:={}^{\circ} P \subseteq H,$$
which is a linear subspace with dimension $n-3$ (note that it has nothing to do with the space $G$ from Section \ref{section:startinductiveadapted},
which we completely forget about from now on).
Hence $\calT$ contains all the operators of the form $\psi \otimes y$ with $y \in H$ and $\psi \in P$.

We introduce the subspace
$$\calU:=\{u \in \calS : \; \im u \subseteq H \quad \text{and} \quad G \subseteq \Ker u\}.$$
Every operator $u \in \calU$ induces a linear operator $\overline{u} : V/G \rightarrow H$,
the set $\overline{\calU}$ of those operators is a linear subspace of $\Hom(V/G,H)$,
and $\overline{\calU}$ is isomorphic to $\calU$ through $u \mapsto \overline{u}$.

Before we move forward, it is useful that we give a matrix interpretation of these spaces.
So, we take a basis $(e_4,\dots,e_n)$ of $G$, extend it to a basis $(e_2,\dots,e_n)$ of $H$, and finally to a basis $(e_1,\dots,e_n)$ of $V$.
If we represent the operators in $\calS$ by matrices in the basis $(e_1,\dots,e_n)$, denote by $\calM$ the corresponding matrix space,
and represent the operators in $\calT$ by matrices in the basis $(e_2,\dots,e_n)$ and denote by $\calM'$ the corresponding matrix space,
then the operators of $\calS^H$ are those of $\calS$ with matrix of the form
$$\begin{bmatrix}
0 & [0]_{1 \times (n-1)} \\
[?]_{(n-1) \times 1} & K(M)
\end{bmatrix},$$
and the linear subspace made of all the $K(M)$ submatrices is precisely $\calM'$.

Next, $\calU$ is the space consisting of all the operators $u$ in $\calS$ with matrix in the previous basis of the form
$$M=\begin{bmatrix}
[0]_{1 \times 3} & [0]_{1 \times (n-3)} \\
J(M) & [0]_{(n-1) \times (n-3)} \\
\end{bmatrix} \quad \text{with $J(M) \in \Mat_{n-1,3}(\F)$,}$$
and the set consisting of all the $J(M)$ matrices is the matrix space that represents $\overline{\calU}$
in the bases $(\overline{e_1},\overline{e_2},\overline{e_3})$ and $(e_2,\dots,e_n)$
(where $\overline{e_i}$ stands for the coset of $e_i$ modulo $G$).
The fact that $\calT$ is a hurdle associated with $P$ then means that, for all $N \in \mathfrak{sl}_2(\F)$ and all $C \in \Mat_{n-2,2}(\F)$,
the space $\calM$ contains a matrix of the form
$$\begin{bmatrix}
0 & [0]_{1 \times 2} & [0]_{1 \times (n-3)} \\
[?]_{2 \times 1} & N & [0]_{2 \times (n-3)} \\
[?]_{(n-3) \times 1} & C & [0]_{(n-3) \times (n-3)}
\end{bmatrix}.$$
In turn, this shows that for all such $N$ and $C$, the matrix space that represents $\overline{\calU}$ in the bases $(\overline{e_1},\overline{e_2},\overline{e_3})$ and $(e_2,\dots,e_n)$
contains a matrix of the form
$$\begin{bmatrix}
[?]_{2 \times 1} & N \\
[?]_{(n-3) \times 1} & C
\end{bmatrix}$$
(and might of course contain several such matrices).

In what follows, we will avoid using matrices when they are cumbersome, but the previous matrix vision will be very useful for the reader to visualize what we are doing.

We move forward with the next point, which will be used repeatedly in what follows:

\begin{claim}\label{claim:tracezero}
All the elements of $\calU$ have trace zero.
\end{claim}

\begin{proof}
Indeed, every element $u \in \calU$ maps into $H$ and hence induces an endomorphism $u_{H} \in \calT$,
which vanishes on $G$. By point (c) of Lemma \ref{lemma:splittinghurdles}, we
obtain that $\tr(u_H)=0$, and hence $\tr(u)=\tr(u_H)=0$.
\end{proof}

We will now work to decipher the structure of the trace-orthogonal space $\overline{\calU}^\bot \subseteq \Hom(H,V/G)$ of $\overline{\calU}$.

\begin{claim}\label{claim:dimUbarbot3}
One has $\dim \overline{\calU}^\bot \geq 3$.
\end{claim}

\begin{proof}
To start with, we consider the restriction
$\calU':=\{u_H \mid u \in \calU\} \subseteq \End(H)$.
We have just seen that all the operators in $\calU'$ have trace zero, and it follows that
$\calU'$ does not contain all the endomorphisms of $H$ that vanish on $G$.
Setting $\overline{\calU'}:=\{\overline{u}_{|H/G} \mid \overline{u} \in \overline{\calU}\} \subseteq \Hom(H/G,H)$,
this means that $\overline{\calU'} \neq \Hom(H/G,H)$, and hence $\dim \overline{\calU'} <\dim(H/G) \dim H$.

Next, we observe that the kernel of $u \in \calU \mapsto \overline{u}_{|H/G}$ is precisely
$f_1 \otimes W_1$, so its dimension is at most $\left\lfloor \frac{n}{2}\right\rfloor$.
Combining the previous two facts yields
$$\dim \overline{\calU}=\dim \calU \leq \left \lfloor \frac{n}{2}\right\rfloor +\dim(H/G) \dim H - 1,$$
hence
$$\dim \overline{\calU}^\bot \geq 1-\left \lfloor \frac{n}{2}\right\rfloor+(\dim(V/G)-\dim(H/G))\dim H
=n-\left \lfloor \frac{n}{2}\right\rfloor \geq \frac{n}{2} \geq \frac{5}{2}$$
and the conclusion follows.
\end{proof}

\begin{claim}\label{claim:targetreduced}
No linear hyperplane of $V/G$ includes the range of every element of $\overline{\calU}^\bot$.
\end{claim}

\begin{proof}
Assume the contrary. Then, by double-orthogonality there is a linear form $\varphi$ on $V/G$ such that $\overline{\calU}$ contains all the operators
in $\Hom(V/G,H)$ that vanish on $\Ker \varphi$.
By composing $\varphi$ with the standard projection $\pi : V \twoheadrightarrow V/G$,
we obtain a nonzero linear form $g:=\varphi \circ \pi \in V^\star$ and we deduce that $g \otimes H \subseteq \calU$.

For all $x \in H$, we deduce from Claim \ref{claim:tracezero} that $g \otimes x$ has trace zero, i.e.,
$g(x)=0$. Hence $g=\lambda f_1$ for some $\lambda \in \F^\times$, and we deduce that $H \subseteq W_1$.
Then $\lfloor \frac{n}{2}\rfloor \geq n-1$, which obviously contradicts the assumption that $n \geq 5$.
\end{proof}

\begin{claim}
The space $\overline{\calU}^\bot$ is intransitive.
\end{claim}

\begin{proof}
Let $x \in H$, and assume that $x \not\in G \cup \underset{i=1}{\overset{n}{\bigcup}} V_i$.
Hence $x$ is not $\calS$-adapted, which yields a nonzero element $g \in V^\star$ such that
$g \otimes x \in \calS$ and $g(x)=0$.
Let us consider the restricted operator $(g \otimes x)_{|H}=g_{|H} \otimes x \in \End(H)$, whose range is included in $\F x$,
and which belongs to $\calT$. By point (a) of Lemma \ref{lemma:splittinghurdles},
the operator $g_{|H} \otimes x$ must leave $G$ invariant, and because $x \not\in G$ it follows that
$g_{|H} \otimes x$ vanishes on $G$. Hence $g_{|G}=0$, and we deduce that $g \otimes x \in \calU$.
We have just shown that $\calU$ contains a rank $1$ operator with range $\F x$. Hence, so does $\overline{\calU}$.
In turn, this shows that $\rk \widehat{x} \leq 2$ for the evaluation operator $\widehat{x} : v \in \overline{\calU}^\bot \mapsto v(x) \in V/G$.

Now, set $m:=\dim \overline{\calU}^\bot$, consider arbitrary bases of $\overline{\calU}^\bot$ and $V/G$, and for $x \in H$ write $M(x) \in \Mat_{3,m}(\F)$ the matrix of $\widehat{x}$
in those bases. Let us choose a subset $I \subseteq \lcro 1,m\rcro$ with cardinality $3$, and denote by
$M(x)_I$ the corresponding $3 \times 3$ minor of $M(x)$ (in which the column indices are taken in $I$).
Set $V'_i:=V_i \cap H$ for all $i \in \lcro 1,n\rcro$. The first part of the proof has shown that $M(x)$ has rank at most $2$ for all
$x \in H \setminus (V'_3 \cup \cdots \cup V'_n \cup G)$.
Hence the mapping $q_I : x \mapsto M(x)_I$, which is a $3$-homogeneous polynomial function, vanishes outside of $V'_3 \cup \cdots \cup V'_n \cup G$.
Just like in the previous section, we examine the list
$(V'_3,\dots,V'_n)$ and observe that it contains only nontrivial linear subspaces of $H$,
no linear hyperplane of $H$ (the greatest possible dimension among such spaces is $\left \lfloor \frac{n-1}{2}\right\rfloor$,
and $n-2-\left \lfloor \frac{n-1}{2}\right\rfloor \geq n-2-\frac{n-1}{2}=\frac{n-3}{2}>0$), and at most two
spaces of dimension $i$ for each $i \in \lcro 1,n-2\rcro$.
Hence the list $(V'_3,\dots,V'_n,G)$ satisfies the assumptions of Lemma \ref{lemma:vanishinghomogeneous} for $d=3$ (thanks to $|\F|>3$).
We conclude that $q_I$ vanishes, and by varying $I$ we obtain that $\rk \widehat{x} \leq 2$ for all $x \in H$, which is the claimed statement.
\end{proof}

We are now in the position to apply Atkinson's theorem to $\overline{\calU}^\bot$, noting that $|\F| \geq 3$ and $\dim (V/G)=3$.
We will prove that if $\overline{\calU}^\bot$ is not primitively intransitive then $\calS$ is a hurdle,
and we will finally use Atkinson's theorem to prove that $\overline{\calU}^\bot$ is not primitively intransitive.

\subsection{The case where $\overline{\calU}^\bot$ is not primitively intransitive}\label{section:casenotprimintrans}

In that case, we choose an intransitivity veil $R$ of $\overline{\calU}^\bot$, which we lift to a subspace $G'$ such that $G \subseteq G' \subseteq V$, and we set $s:=\dim((V/G)/R)$.
Here we have $s \in \{1,2\}$, and the case $s=1$ is immediately ruled out by Claim \ref{claim:targetreduced}.
Hence $s=2$, and by taking $\pi : V/G \twoheadrightarrow (V/G)/R \simeq V/G'$
we find by Atkinson's theorem that $\pi \overline{\calU}^\bot$ has dimension $1$ (it cannot have dimension $0$, as this would yield that every operator in $\overline{\calU}^\bot$
maps into $R$, contradicting Claim \ref{claim:targetreduced}).
By double-orthogonality, this yields a linear hyperplane $\calU'$ of
$\{u \in \Hom(V,H) : u_{|G'}=0\}$ such that $\calU' \subseteq \calU$.
Note in particular that $\dim \calU'=2n-3$.

\begin{claim}\label{claim:G'+H}
One has $G'+H=V$.
\end{claim}

\begin{proof}
Assume the contrary. Then $G' \subseteq H$ because $H$ is a linear hyperplane of $V$.
It follows that $f_1 \otimes H$ is included in $\{u \in \Hom(V,H) : u_{|G'}=0\}$, and
hence $\dim(\calU \cap (f_1 \otimes H)) \geq \dim(\calU' \cap (f_1 \otimes H)) \geq \dim (f_1 \otimes H)-1=n-2$, to the effect that
$\dim W_1 \geq n-2$. Once more, this contradicts the fact that $\dim W_1 \leq \lfloor n/2\rfloor$ and $n \geq 5$.
\end{proof}

\begin{claim}
The space $\calU'$ consists of all the trace-zero endomorphisms of $V$ that map into $H$ and vanish on $G'$.
\end{claim}

\begin{proof}
Set $\calU'':=\{u \in \mathfrak{sl}(V) : \im u \subseteq H \; \text{and}\; u_{|G'}=0\}$.
Given $u \in \calU''$, the induced endomorphism $u_H \in \End(H)$ has trace zero and vanishes on $G$,
and we observe that the mapping $\Phi : u \in \calU'' \mapsto u_H$ is injective thanks to Claim \ref{claim:G'+H}.

Next, we note that the range of $\Phi$ is included in the space $\calX$ of all trace zero endomorphisms of $H$ that vanish on $G$,
which clearly has dimension $2n-3$.

Finally, every $u \in \calU'$ belongs to $\calU$ and hence has trace zero by Claim \ref{claim:tracezero}.
Hence we successively find $\calU' \subseteq \calU''$ and $2n-3 \geq \rk \Phi \geq \dim \calU'' \geq \dim \calU'=2n-3$. We conclude that
$\calU'=\calU''$, as claimed.
\end{proof}

Sine $\calU' \subseteq \calU$, we are precisely in the situation of the next lemma (where $G$ plays the role of the space $G'$ we are considering here).
The proof of this lemma is independent of the induction process, and we postpone it to
Section \ref{section:confinement} so as to streamline the current proof.

\begin{lemma}[Second Confinement Lemma]\label{lemma:confinement2}
Let $V$ be a vector space with dimension $n \geq 3$.
Let $H$ be a linear hyperplane of $V$, and $G$ be a linear subspace of $V$ with codimension $2$ such that $G \not\subseteq H$.
Let $\calS$ be a $2$-spec linear subspace of $\End(V)$ that contains all the trace-zero endomorphisms of $V$ that
vanish on $G$ and map into $H$.
Then either $\calS$ is a hurdle or there exists a linear hyperplane $H'$ of $V$ such that every
vector of $V$ that is not $\calS$-adapted belongs to the union $G \cup H \cup H'$.
\end{lemma}

Let us admit the validity of this lemma until the end of this section.

Assume that $\calS$ is not a hurdle. Then we obtain a linear hyperplane $H'$ of $V$ such that all vectors
that are not $\calS$-adapted belong to $H \cup H' \cup G'$.
Remembering the starting assumption on the position of the $\calS$-adapted vectors, we
deduce that the (proper) linear subspaces $V_1,\dots,V_n,G',H,H'$ cover $\calS$.
Here $\dim V_n \leq n-2$ and hence there are at most three subspaces in this sequence for each given dimension in $\lcro 1,n-1\rcro$.
By the Covering Lemma, we obtain a contradiction.

We deduce that $\calS$ is a hurdle.

\subsection{Proof that $\overline{\calU}^\bot$ is not primitively intransitive}\label{section:lastcaseprimitively}

In this section, we assume that $\overline{\calU}^\bot$ is primitively intransitive and we prove that this leads to a contradiction.
As $\dim \overline{\calU}^\bot \geq 3$ (Claim \ref{claim:dimUbarbot3}), we can apply the second point in Atkinson's theorem
to find that $\dim \overline{\calU}^\bot=3$ and that there is a right-nondegenerate bilinear form
$b : H \times (V/G) \rightarrow \F$ such that
$$\forall x \in H, \; \forall v \in \overline{\calU}^\bot, \; b(x,v(x))=0.$$
Denote by
$$L:=\{x \in H : \; b(x,-)=0\}$$
the left-radical of $b$, and note that $\dim L=\dim H-\dim (V/G)=n-4$ because $b$ is right-nondegenerate.
By coming back to the proof of Claim \ref{claim:dimUbarbot3}, we could also infer that $n \leq 6$, but this property will turn out to be useless.

Let $v \in \overline{\calU}^\bot$. Polarizing the previous identity yields $b(x,v(x'))+b(x',v(x))=0$ for all $(x,x') \in H^2$,
and we deduce that $v$ maps $L$ into the right-radical of $b$, which equals zero.
In other words, all the elements of $\overline{\calU}^\bot$ vanish on $L$.
By double-orthogonality, it follows that $\calU$ contains all the elements of $\End(V)$ that map into $L$ and vanish on $G$.
We shall deduce:

\begin{claim}
The space $L$ is included in $G$.
\end{claim}

\begin{proof}
Let $x \in L$.
Let $f \in V^\star$ vanish on $G$. Then $f \otimes x$ belongs to $\calU$, so by Claim \ref{claim:tracezero} its trace is zero. Hence $f(x)=0$. Varying $f$ yields $x \in G$.
\end{proof}

Now we consider the induced nondegenerate bilinear form
$$\overline{b} : (H/L) \times (V/G) \rightarrow \F.$$
We take a nonzero vector $e_4 \in G \setminus L$, we extend it to a linearly independent triple $(e_2,e_3,e_4)$ such that $\Vect(e_2,e_3,e_4) \oplus L=H$,
and finally we choose $e_1 \in V \setminus H$. Note that $H=\Vect(e_2,e_3) \oplus G$.
Now, we consider the respective cosets $\overline{e_1},\overline{e_2},\overline{e_3}$ in $V/G$ and the respective cosets
$\widetilde{e_2},\widetilde{e_3},\widetilde{e_4}$ in $H/L$, and the mixed Gram matrix $Q \in \GL_3(\F)$  of the pair
$\bigl((\widetilde{e_2},\widetilde{e_3},\widetilde{e_4}),(\overline{e_1},\overline{e_2},\overline{e_3})\bigr)$ of bases for $\overline{b}$
(i.e., the matrix $(b(\widetilde{e_{i+1}},\overline{e_j}))_{1 \leq i,j \leq 3}$).

\begin{claim}
For some $\lambda \in \F$, the matrix $Q$ has the form
$$Q=\begin{bmatrix}
? & 0 & \lambda \\
? & \lambda & 0 \\
? & 0 & 0
\end{bmatrix}$$
\end{claim}

\begin{proof}
We use the key observation that all the elements in $\calU$ are trace zero endomorphisms, which has the consequence that
$\overline{\calU}^\bot$ contains the linear operator $v : H \rightarrow V/P$ that vanishes on $G$ and maps each one of $e_2$ and $e_3$ to its coset modulo $G$.
Since $v$ is $b$-alternating, we deduce that the product $Q \times M$ is alternating for
$M:=\begin{bmatrix}
0 & 0 & 0 \\
1 & 0 & 0 \\
0 & 1 & 0
\end{bmatrix}$. The claimed statement follows easily.
\end{proof}

Now, we write $Q=(q_{i,j})_{i,j}$. Since $Q$ is invertible, we have $q_{1,3} \neq 0$ and $q_{2,2} \neq 0$.
Then, we successively change the choice of $e_1,e_2,e_3$ and $e_4$ as follows:
\begin{itemize}
\item We replace $e_1$ with $e_1-q_{2,1} q_{2,2}^{-1} e_2- q_{1,1} q_{1,3}^{-1} e_3$, which reduces the situation to the one where $q_{2,1}=q_{1,1}=0$;
\item Then we replace $e_4$ with $q_{3,1}^{-1} e_4$, which further reduces the situation to the one where $q_{3,1}=1$, while keeping all the previous assumptions;
\item Finally, we replace $e_3$ with $q_{2,2}^{-1} e_3$, which further reduces the situation to the one where $q_{2,2}=1$, while keeping all the previous assumptions.
\end{itemize}
After these changes of basis vectors, we are reduced to the situation where
$$Q=\begin{bmatrix}
0 & 0 & 1 \\
0 & 1 & 0 \\
1 & 0 & 0
\end{bmatrix}.$$
In that situation, we take an arbitrary basis $(e_5,\dots,e_n)$ of $L$, and we conclude that every operator in $\overline{\calU}^\bot$
is represented, in the bases $(e_2,\dots,e_n)$ and $(\overline{e_1},\overline{e_2},\overline{e_3})$, by a matrix of the form
$$\begin{bmatrix}
\alpha & \gamma & 0 & [0]_{1 \times (n-4)} \\
\beta & 0 & \gamma & [0]_{1 \times (n-4)}\\
0 & \beta & \alpha & [0]_{1 \times (n-4)}
\end{bmatrix} \quad \text{with $(\alpha,\beta,\gamma) \in \F^3$}.$$
Note that from $\dim \overline{\calU}^\bot \geq 3$ we could infer that these are exactly the matrices that represent the elements of $\overline{\calU}^\bot$
in the said bases, but we will not need this precision in what follows.

Now, we completely forget about the $b$ form, and we find that the subspace $\calU$ contains all the operators that are represented in
the basis $(e_1,\dots,e_n)$ by a matrix of the form
$$\begin{bmatrix}
0 & 0 & 0 & [0]_{1 \times (n-3)} \\
a & \lambda & y & [0]_{1 \times (n-3)} \\
b & x & \lambda & [0]_{1 \times (n-3)} \\
c & b & a & [0]_{1 \times (n-3)} \\
[?]_{(n-4) \times 1} &  [?]_{(n-4) \times 1} &  [?]_{(n-4) \times 1} & [0]_{(n-4) \times (n-3)}
\end{bmatrix}.$$

Then the next and final lemma, to be proved in the final section of the article,
yields two linear hyperplans $H_1$ and $H_2$ of $V$ whose union contains all the vectors of $V$
that are not $\calS$-adapted.
Hence the subspaces $V_1,\dots,V_n,H_1,H_2$ cover $V$.
Yet this contradicts the Covering Lemma, like in the situation considered in the end of Section \ref{section:casenotprimintrans}.

\begin{lemma}[Third Confinement Lemma]\label{lemma:confinement3}
Let $\calM$ be a $2$-spec subspace of $\Mat_n(\F)$ for some $n \geq 5$.
Assume that $\calM$ contains every matrix of the form
$$\begin{bmatrix}
0 & 0 & 0 & [0]_{1 \times (n-3)} \\
a & \lambda & y & [0]_{1 \times (n-3)} \\
b & x & \lambda & [0]_{1 \times (n-3)} \\
c & b & a & [0]_{1 \times (n-3)} \\
[?]_{(n-4) \times 1} &  [?]_{(n-4) \times 1} &  [?]_{(n-4) \times 1} & [0]_{(n-4) \times (n-3)}
\end{bmatrix}.$$
Then the non-$\calM$-adapted vectors all belong to the union $H_1\cup H_2$ of two specific linear hyperplanes $H_1$ and $H_2$ of $\F^n$
(precisely, the linear hyperplanes $\{0\} \times \F^{n-1}$ and $\F^2 \times \{0\} \times \F^{n-3}$).
\end{lemma}

\subsection{The proof for $2$-spec subspaces}

So far, we have only dealt with $1^\star$-spec subspaces!
Fortunately, adapting the previous proof to $2$-spec subspaces is very easy.

We shall prove Theorem \ref{theorem:2specwithadapted} directly.

First of all, the cases $n=3$ and $n=4$ have already been dealt with in Section \ref{section:largefields}
(note that we require no adapted vector for $n=3$, although the existence of such a vector can be proved).

So, from now on we take an $n$-dimensional vector space $V$, with $n \geq 5$, and a $2$-spec linear subspace
$\calS$ of $\End(V)$. We have to prove that $\calS$ is a hurdle or that it has an adapted vector.

Instead of using an induction hypothesis, we directly rely upon Proposition \ref{prop:inductionprop}.

First of all, we deduce from Lemma \ref{lemma:diagonalzero} that
there exists a nonzero vector $x$ in $V$ such that $x^\bot \otimes x$ is not included in $\calS$.
As $\dim \overline{\calS_{x,0}} \leq \dbinom{n-1}{2}+2$ by Theorem \ref{theorem:maintheodim1starspec}, we find
$$\dim \calS \leq \dbinom{n-1}{2}+2+\dim(\calS \cap (x^\bot \otimes x))+n \leq \dbinom{n}{2}+n+1.$$
Next, we assume that there is a $3$-complex of $V$ whose union contains all the weakly $\calS$-adapted vectors.
Combining this assumption with point (ii) in Proposition \ref{prop:inductionprop} and the same strategy as in Section \ref{section:inductiveweaklyadapted} -- which works because the $\calT$ space
is actually $1^\star$-spec -- we recover that
$$\dim \calS \geq n+n \left\lfloor 2n/3\right\rfloor.$$
Yet
$$n \left\lfloor 2n/3\right\rfloor- \dbinom{n}{2}-1 \geq n\,\frac{2n-2}{3}-n\,\frac{n-1}{2}-1
= \frac{n(n-1)}{6}-1>0$$
because $n \geq 4$. This contradicts the previous rough upper-bound, and we deduce that
the union of a $3$-complex of $V$ cannot contain all the weakly $\calS$-adapted vectors.
In particular, there is at least one such vector, and as usual we derive the improved inequality
$$\dim \calS \leq \dbinom{n}{2}+4.$$
In the final round of proof, we assume that the union of some $2$-complex of $V$ contains all the $\calS$-adapted vectors.
Once more, it is critical that we can prove that $n+n \left \lfloor \frac{n}{2}\right\rfloor$ is greater than the previous bound
$\dbinom{n}{2}+4$, but this is easily obtained by noting that
$$n+n \left \lfloor \frac{n}{2}\right\rfloor \geq n+\frac{n(n-1)}{2} > 4+\dbinom{n}{2}$$
since $n > 4$.
Hence we can follow the rest of the proof from Sections \ref{section:startinductiveadapted} to \ref{section:lastcaseprimitively},
by relying upon the validity of point (iii) of Proposition \ref{prop:inductionprop} for the subspace $\calT$, which is $1^\star$-spec
(even though $\calS$ is only assumed to be $2$-spec here).

This proves that either $\calS$ is a hurdle or the union of $2$-complex of $V$ cannot contain all the $\calS$-adapted vectors.
In the former case, we obtain the inequality $\dim \calS \leq \dbinom{n}{2}+3$ by applying Lemma \ref{lemma:ineqdimhurdles};
in the latter case, there is at least one $\calS$-adapted vector, and we conclude that
$\dim \calS \leq \dbinom{n}{2}+3$ by applying Theorem \ref{theorem:maintheodim1starspec}, as explained in the end of Section \ref{section:strategy}.
Hence Theorem \ref{theorem:2specwithadapted} is proved, and in particular we have also proved Theorem \ref{theorem:maintheodim2spec}.

\section{Proof of the remaining confinement lemmas}\label{section:confinement}

As seen earlier, the previous results rely, for $n \geq 5$, on two confinement lemmas
(Lemmas \ref{lemma:confinement2} and \ref{lemma:confinement3}) of which we will now give
proofs, independently of the induction process that we have performed
(however, the very last step of the proof of Lemma \ref{lemma:confinement3} will use the case $n=3$ in Theorem \ref{theorem:maintheodim1starspec} to speed things up).

\subsection{A basic lemma}

To start with, we need an additional result on $3$-by-$3$ matrices, which will be used in the proof of both confinement lemmas.

\begin{lemma}\label{lemma:lastblock}
Let $A \in \Mat_3(\F)$ be a rank $1$ and trace zero matrix.
Assume that the sum of $A$ with any matrix of the form
$$\begin{bmatrix}
N & [0]_{2 \times 1} \\
[0]_{1 \times 2} & 0
\end{bmatrix} \quad \text{with $N \in \mathfrak{sl}_2(\F)$,}$$
 is $2$-spec. Then the last column of $A$ is zero or the last row of $A$ is zero.
\end{lemma}

\begin{proof}
We assume that the last column of $A$ is nonzero and the last row of $A$ is nonzero.

We use an operator interpretation of the situation.
Set $P:=\F^2 \times \{0\}$ and denote by $(e_1,e_2,e_3)$ the standard basis of $\F^3$.
We consider the operator $u_0 : X \in \F^3 \mapsto AX \in \F^3$, which has rank $1$ and trace zero,
and the linear subspace $S$ of all $v \in \mathfrak{sl}(\F^3)$ such that $v(e_3)=0$ and $\im v \subseteq P$.
The assumptions mean that the subspace $\F u_0+S$ is $2$-spec.
We have $u_0(e_3) \neq 0$, hence $\im u_0=\F u_0(e_3)$ and the assumption that the last row of $A$ is nonzero yields $u_0(e_3) \not\in P$.
Then we decompose $u_0(e_3)=y+\delta e_3$ for some $\delta \in \F^\times$ and some $y \in P$.
If $y=0$ then $\delta$ would be an eigenvalue of $u_0$, and since $u_0$ has rank $1$ we would find $\delta=\tr(u_0)=0$.
Hence $y \neq 0$. Then we extend $y$ into a basis $(e'_1,e'_2)$ of $P$ with $e'_2=y$, and we consider the matrices that represent
the elements of $u_0+S$ in the basis $(e'_1,e'_2,e_3)$. This essentially replaces $A$ with a matrix in which the last column is
$\begin{bmatrix}
0 & 1 & \delta
\end{bmatrix}$.
\emph{From now on we consider only this reduced situation.} Since $A$ has rank $1$ and trace zero, we now have
$$A=\begin{bmatrix}
0 & 0 & 0 \\
? & \delta & 1 \\
\alpha & \beta & \delta
\end{bmatrix} \quad \text{for some $(\alpha,\beta) \in \F^2$.}$$
Because $A$ has rank $1$, we find $\beta=\delta^2$ by considering the lower-right $2 \times 2$ minor.

Let $x,y,\lambda$ belong to $\F$. Then the matrix
$$M=\begin{bmatrix}
\lambda & y & 0 \\
x & \delta+\lambda & 1 \\
\alpha & \beta & \delta
\end{bmatrix}$$
is $2$-spec, and we compute that the respective coefficients of the characteristic polynomial of $M$ on $t$ and $1$ are
$$c_2(M)=xy+\beta+(\lambda+\delta)^2+\lambda \delta=xy+\lambda^2+\lambda \delta$$
and
$$c_3(M)=\alpha y+\delta xy+\lambda(\delta^2+\lambda \delta+\beta)=\alpha y +\delta xy+\lambda^2 \delta.$$
We combine these to see that
$$c_3(M)+\delta c_2(M) = \alpha y+\delta^2 \lambda.$$
Of course we also have $\tr(M)=0$.
Now, we choose two distinct nonzero elements $z_1,z_2$ in $\F$, and we note that $(t-z_1)(t-z_2)(t-(z_1+z_2))$ has trace zero
and three distinct roots in $\F$.
Let us write this polynomial as $t^3+pt-q$.
Then, we set $y:=1$, and since $\delta \neq 0$ we can adjust $\lambda$ so that $c_3(M)+\delta c_2(M)=q+\delta p$,
and finally we adjust $x$ so that $c_2(M)=p$. With these choices, we have $\chi_M=(t-z_1)(t-z_2)(t-(z_1+z_2))$,
which contradicts the assumption that $M$ is $2$-spec.
\end{proof}

\subsection{Proof of the Second Confinement Lemma}

Let us recall the situation. We take a vector space $V$ with dimension $n \geq 3$, a linear hyperplane $H$ of $V$ and
a linear subspace $G$ of $V$ with codimension $2$ such that $G \not\subseteq H$.
Finally, we let $\calS$ be a $2$-spec linear subspace of $\End(V)$ that contains all the trace-zero endomorphisms of $V$ that
vanish on $G$ and map into $H$. We set $P:=G^\circ$ (a $2$-dimensional linear subspace of the dual space $V^\star$).
The latter assumption means that $\varphi \otimes x \in \calS$ for all $x \in H$ and all $\varphi \in P$ such that $\varphi(x)=0$.

We assume that $\calS$ is not hurdle, and we will construct a linear hyperplane $H'$ of $V$ such that every
vector of $V$ that is not $\calS$-adapted belongs to the union $G \cup H \cup H'$.
Throughout, we fix a vector $z \in G \setminus H$.

\vskip 2mm
\noindent \textbf{Step 1. The intersection $\calS \cap (P \otimes z)$ has dimension at most $1$.} \\
To see this, a matrix interpretation will help. Note that $G':=G \cap H$ is a linear hyperplane of $G$
and that $G' \oplus \F z=G$. Then, we choose a basis $(e_1,\dots,e_{n-3})$ of $G'$, extend it into a basis $(e_1,\dots,e_{n-3},e_{n-1},e_n)$ of $H$,
and finally we note that $(e_1,\dots,e_{n-3},z,e_{n-1},e_n)$ is a basis of $V$ and $(e_1,\dots,e_{n-3},z)$ is a basis of $G$.
Denote by $\calM$ the matrix space that represents $\calS$ in the basis $(e_1,\dots,e_{n-3},z,e_{n-1},e_n)$. The starting assumption means that
$\calM$ contains every trace zero matrix of the form
$$\begin{bmatrix}
[0]_{(n-3) \times (n-2)} & [?]_{(n-3) \times 2} \\
[0]_{1 \times (n-2)} & [0]_{1 \times 2} \\
[0]_{2 \times (n-2)} & [?]_{2 \times 2}
\end{bmatrix}.$$
Now, if $\calS \cap (P \otimes z)$ had dimension more than $1$ then $P \otimes z \subseteq \calS$, and the matrix interpretation of this is that
$\calM$ would contain every matrix of the form
$$\begin{bmatrix}
[0]_{(n-3) \times (n-2)} & [0]_{(n-3) \times 2} \\
[0]_{1 \times (n-2)} & [?]_{1 \times 2} \\
[0]_{2 \times (n-2)} & [0]_{2 \times 2}
\end{bmatrix}.$$
But then, summing the matrices of the above type would yield that $\calM$ contains every trace zero matrix of the form
$$\begin{bmatrix}
[0]_{(n-3) \times (n-2)} & [?]_{(n-3) \times 2} \\
[0]_{1 \times (n-2)} & [?]_{1 \times 2} \\
[0]_{2 \times (n-2)} & [?]_{2 \times 2}
\end{bmatrix},$$
thereby contradicting the assumption that $\calS$ is not a hurdle. Therefore $\dim \bigl(\calS \cap (P \otimes z)\bigr) \leq 1$.

\vskip 2mm
\noindent \textbf{Step 2. Defining $H'$.} \\
Now, either $\calS \cap (P \otimes z)=\F (\theta \otimes z)$ for some $\theta \in P \setminus \{0\}$,
in which case we set $H':=\Ker \theta$, or $\calS \cap (P \otimes z)=\{0\}$, and then we set
$H':=H$ and take an arbitrary $\theta \in P \setminus\{0\}$ (and note that $\theta(z)=0$ in any case).

\vskip 2mm
\noindent \textbf{Step 3. A basic property of some non-$\calS$-adapted vectors.} \\
Let $\psi \in P \setminus \{0\}$ and $x \in V \setminus (G \cup H)$ be
such that $\psi \otimes x \in \calS$ and $\psi(x)=0$. Then we shall prove that $x \in H'$.
To see this, we split $x=\lambda z+y$ for some $\lambda \in \F^\times$ and some $y \in H$.
Then $\psi(z)=0$, and hence $\psi(y)=0$. It follows that
$\psi \otimes y$ has trace zero and maps into $H$ with $\psi \in P$, so by our starting assumption it must belong to $\calS$.
By linearly combining, we deduce that $\psi \otimes z \in \calS$. We are then in the first situation stated in Step 2,
and we recover $\psi=\mu \theta$ for some $\mu \in \F^\times$. Then we find $\theta(x)=\mu^{-1} \psi(x)=0$,
that is $x \in H'$.

Our aim in the remainder of the proof is to extend this property to all the trace zero tensors $\psi \otimes x$ that belong to $\calS$
and in which $x \in V \setminus (G \cup H)$ (and not only those for which $\psi \in P \setminus \{0\}$). To obtain this, we will find it more convenient to work with the transposed operator space
$\calS^t$.

\vskip 2mm
\noindent \textbf{Step 4. Dual translation of the situation.} \\
We will write $U:=V^\star$ for greater convenience, as well as $D:=H^\circ$, and we recall that $P=G^\circ$.
Note that $D$ is a $1$-dimensional linear subspace of $V^\star$ (whereas $P$ is a $2$-dimensional subspace). The assumption $G \not\subseteq H$
translates into $D \not\subseteq P$. We will need the following observation at some point:
the mapping $\varphi \in D^\circ \mapsto \varphi_{|P} \in P^\star$ is surjective. Indeed, if not
its range is included in a $1$-dimensional linear subspace of $P^\star$, which we can then interpret as the dual orthogonal of some nonzero vector $z'$ of $P$,
which yields that $D^\circ \subseteq (z')^\bot$ and hence $z' \in D$. But then as $D$ has dimension $1$, $D=\F z'$ would be included in $P$, which we have just disproved.

Now, we come back to $\calS^t$, which is a $2$-spec linear subspace of $\End(V^\star)$.
The starting assumptions mean that $\calS^t$ contains $\varphi \otimes y$ for all $y \in P$ and all $\varphi \in D^\circ$ such that $\varphi(y)=0$.

\vskip 2mm
\noindent \textbf{Step 5. Preparing the final step.} \\
Let $y$ be a non-$\calS$-adapted vector outside of $G \cup H$.
Setting $\varphi_0:=\mathfrak{i}_V(y)$, this translates into the existence of some vector $x \in U \setminus \{0\}$
such that $\varphi_0(x)=0$ and $\varphi_0 \otimes x \in \calS^t$, and the assumption on $y$ means that $\varphi_0 \not\in P^\circ \cup D^\circ$.
Our goal in the next three steps is to prove that $x \in P$, and then by Step 3 we will deduce that $y\in H'$.

So, we assume that $x \not\in P$ and we show that it leads to a contradiction (which will be obtained in the end of Step 7 below).
Since $\varphi_0 \not\in P^\circ$, we can choose a basis $(f_1,f_2)$ of $P$ such that
$\varphi_0(f_1)=0$ and $\varphi_0(f_2)=1$. Then $\varphi_0 \otimes x$ vanishes at $f_1$ and maps $f_2$ to $x$.

Next, the surjectivity of $\varphi \in D^\circ \mapsto \varphi_{|P} \in P^\star$ allows us to find
$\psi \in D^\circ$ such that $\psi(f_1)=1$ and $\psi(f_2)=0$, and hence $\psi \otimes f_2 \in \calS^t$,
and $\psi \otimes f_2$ vanishes at $f_2$ and maps $f_1$ to $f_2$. Now, we note that
$$u:=\varphi_0 \otimes x+\psi \otimes f_2$$ 
maps $f_1$ to $f_2$ and $f_2$ to $x$.
Finally, since $u$ has rank at most $2$ we have just proved that its range is $\Vect(f_2,x)$, and
hence $\Vect(f_1,f_2,x)$ is invariant under $u$. We deduce that $\Vect(f_1,f_2,x)$ is also invariant under any sum of $u$ with tensors of the form
$\gamma \otimes z$ with $z \in P$ and $\gamma \in D^\circ$ such that $\gamma(z)=0$ (as their range is included in $P=\Vect(f_1,f_2)$).

\vskip 2mm
\noindent \textbf{Step 6. Proof that $x \not\in D+P$.} \\
Assume on the contrary that $x \in D+P$. Choose $f_3 \in D \setminus \{0\}$.
Since $x \not\in P$, we have $\Vect(f_1,f_2,x)=D+P=\Vect(f_1,f_2,f_3)$.
Denote by $A$ the matrix of the endomorphism of $D+P$ induced by $\varphi_0 \otimes x$, in the basis $(f_1,f_2,f_3)$.
Note that $A$ has rank $1$ and trace zero. The corresponding matrices of endomorphisms induced by the tensors of the form
$\theta \otimes z$, with $z \in P$ and $\theta \in D^\circ$ such that $\theta(z)=0$, are exactly the matrices of the form
$N \oplus 0_1$ with $N \in \mathfrak{sl}_2(\F)$, and hence we can use Lemma \ref{lemma:lastblock}.
Yet we note that $(\varphi_0 \otimes x)(f_2)=x$ and hence the last row of $A$ is nonzero, whereas
$(\varphi_0 \otimes x)(f_3)=\varphi_0(f_3)\, x$ is nonzero because $\varphi_0(f_3) \neq 0$ (indeed, we have $\varphi_0 \not\in D^\circ$ from the start).
Hence the last column of $A$ is nonzero. This contradicts Lemma \ref{lemma:lastblock}, and we conclude that
$x \not\in D+P$.

\vskip 2mm
\noindent \textbf{Step 7. The final contradiction.} \\
Now we have $x \not\in D+P$, to the effect that $D^\circ \cap P^\circ \not\subseteq x^\bot$.
which yields a linear form $f \in D^\circ \cap P^\circ$ such that $f(x) \neq 0$. As a consequence $\calS^t$ contains $f \otimes z$ for all $z \in P$.
By adding these tensors to $u$ and by considering the induced endomorphisms of $\Vect(f_1,f_2,x)$, we obtain representing matrices in the basis $(f_1,f_2,x)$ of the form
$$\begin{bmatrix}
? & ? & \alpha \\
1 & ? & \beta \\
0 & 1 & ?
\end{bmatrix}$$
where the question marks are fixed values (depending only on $\varphi_0 \otimes x$) and the parameters $\alpha$ and $\beta$
can be chosen at will. Moreover, these matrices must have trace zero.
Then, we pick distinct elements $z_1,z_2$ in $\F \setminus \{0\}$, we note that the polynomial $(t-z_1)(t-z_2)(t-(z_1+z_2))$
has trace zero, and hence we can use the Choice Lemma to see that $\alpha$ and $\beta$ can be adjusted so that the characteristic polynomial
of the above matrix equals $(t-z_1)(t-z_2)(t-(z_1+z_2))$, which has three distinct roots in $\F$. This contradicts the assumption that $\calS^t$ is $2$-spec.

This final contradiction yields $x \in P$, and by Step 3 we conclude that $y \in H'$. Hence we have proved that all the $\calS$-adapted vectors
belong to $G \cup H \cup H'$. The Second Confinement Lemma is now proved.

\subsection{Proof of the Third Confinement Lemma}

Here we consider an integer $n \geq 5$ and a matrix space $\calM \subseteq \Mat_n(\F)$ that satisfies the assumptions of the
Third Confinement Lemma. We take an arbitrary vector space $V$ and a basis $(e_1,\dots,e_n)$ of $V$, and we consider the operator space
$\calS \subseteq \End(V)$ that is represented by $\calM$ in $(e_1,\dots,e_n)$. We denote by $(e_1^\star,\dots,e_n^\star)$ the dual basis.

One important key is that a simple basis change is possible without fundamentally altering our assumptions.
First of all, we can allow a less precise assumption, where we only suppose that for some nonzero scalar $\eta \in \F^\times$
the space $\calM$ contains all the matrices of the form
$$\begin{bmatrix}
0 & 0 & 0 & [0]_{1 \times (n-3)} \\
\eta a & \lambda & y & [0]_{1 \times (n-3)} \\
\eta b & x & \lambda & [0]_{1 \times (n-3)} \\
c & b & a & [0]_{1 \times (n-3)} \\
[?]_{(n-4) \times 1} &  [?]_{(n-4) \times 1} &  [?]_{(n-4) \times 1} & [0]_{(n-4) \times (n-3)}
\end{bmatrix} \quad \text{with $(a,b,c,\lambda,x,y)\in \F^6$.}$$

Now, let $P \in \GL_2(\F)$. If we let $a,b$ in $\F$ and set $L:=\begin{bmatrix}
b & a
\end{bmatrix}$, then $\begin{bmatrix}
a \\
b
\end{bmatrix}=K L^T$ for $K:=\begin{bmatrix}
0 & 1 \\
1 & 0
\end{bmatrix}$, which is alternating.
Then replacing
$\calM$ with $\calM':=(\alpha I_1 \oplus P \oplus I_{n-3}) \calM (\alpha I_1 \oplus P \oplus I_{n-3})^{-1}$
leaves the looser assumption entirely unchanged, owing to the following observations:
\begin{enumerate}[(i)]
\item The space $\mathfrak{sl}_2(\F)$ is invariant under conjugation;
\item One has $P K P^T=\gamma K$ for some $\gamma \in \F^\times$, because $PKP^T$ is alternating (being congruent to an alternating matrix) and nonzero.
\end{enumerate}
This means that we can change the pair $(e_2,e_3)$ at will, as long as it still spans the same $2$-linear subspace
and the other basis vectors are unchanged.
We can also change $e_1$ at will by replacing it with another element of $\F^\times e_1$.

That being said, for the first part of the proof it will be more convenient to work with the transposed operator space $\calS^t$.
Let $y \in V \setminus \{0\}$ be non-adapted to $\calS$. Assume that $y$ belongs to neither
$H_1:=\Vect(e_1,e_2,e_4,\dots,e_n)$ nor $H_2:=\Vect(e_2,e_3,\dots,e_n)$.
We set $U:=V^\star$ and $\varphi:=\mathfrak{i}_V(y) \in U^\star$.
Then we have a nonzero vector $x \in U$ such that $\varphi \otimes x \in \calS^t$ and $\varphi(x)=0$.

The restriction of $\varphi$ to the subspace $\Vect(e_2^\star,e_3^\star)$ does not vanish, otherwise
$y \in H_1$. Hence we can perform a change of basis of the kind we have explained earlier so as to reduce the situation to the one where
$\varphi(e_2^\star)=0$. We immediately derive that $\varphi(e_3^\star) \neq 0$.
By scaling $\varphi$ we can actually assume that $\varphi(e_3^\star)=1$.
Finally we note that $\varphi(e_1^\star) \neq 0$ because $y \not\in H_2$, and by rescaling $e_1$ we can assume that
$\varphi(e_1^\star)=1$. To sum up:
$$\varphi(e_1^\star)=\varphi(e_3^\star)=1 \quad \text{and} \quad \varphi(e_2^\star)=0.$$
There will be three main steps from where we are.
We will successively prove that $x \in \Vect(e_1^\star,e_2^\star,e_3^\star,e_4^\star)$,
that $x \in \Vect(e_1^\star,e_2^\star,e_3^\star)$, and in the last step we will show that the latter point leads to a contradiction.

\vskip 3mm
\noindent \textbf{Step 1: Proving that $x \in \Vect(e_1^\star,e_2^\star,e_3^\star,e_4^\star)$.} \\
We assume on the contrary that $x \not\in \Vect(e_1^\star,e_2^\star,e_3^\star,e_4^\star)$.

We will use the same line of reasoning as in the proof of the Second Confinement Lemma. First of all we note that
$\calS$ contains $e_3^\star \otimes e_2$, and hence $\varphi_0 :=\mathfrak{i}_V(e_2)$
is such that $\varphi_0 \otimes e_3^\star \in \calS^t$, $\varphi_0(e_2^\star)=1$ and $\varphi_0(e_3^\star)=0$.
Then we consider the sum  $\varphi_0 \otimes e^\star_3+\varphi \otimes x$ and we note that it leaves $\Vect(e_2^\star,e_3^\star,x)$ invariant and that
the resulting matrix in the basis
$(e_2^\star,e_3^\star,x)$ is of the form $A=\begin{bmatrix}
0 & 0 & ? \\
1 & 0 & ? \\
0 & 1 & 0
\end{bmatrix}$. Next, since $x \not\in \Vect(e_1^\star,e_2^\star,e_3^\star,e_4^\star)$ we find $\theta \in \{e_1^\star,e_2^\star,e_3^\star,e_4^\star\}^\circ$ such that $\theta(x)=1$.
Then $\theta=\mathfrak{i}_V(z)$ for some $z \in \Vect(e_5,\dots,e_n)$, and it is then clear from the starting assumptions
on $\calS$ that $e_2^\star \otimes z$ and $e_3^\star \otimes z$ belong to $\calS$, to the effect that
$\theta \otimes e_2^\star$ and $\theta \otimes e_3^\star$ belong to $\calS^t$.
These endomorphisms leave $\Vect(e_2^\star,e_3^\star,x)$ invariant, and the corresponding resulting matrices in the basis $(e_2^\star,e_3^\star,x)$
are the unit matrices $E_{1,3}$ and $E_{2,3}$. Since $\calS^t$ is $2$-spec, by linearly combining the matrices $A,E_{1,3},E_{2,3}$ we deduce that all
$3$-by-$3$ companion matrices with trace zero are $2$-spec, which is obviously false
(as used several times earlier, there is a trace zero polynomial of degree $3$ with three distinct roots in $\F$).
We deduce that $x \in \Vect(e_1^\star,e_2^\star,e_3^\star,e_4^\star)$.

\vskip 3mm
\noindent \textbf{Step 2: Proving that $x \in \Vect(e_1^\star,e_2^\star,e_3^\star)$.} \\
This is the most difficult part of the proof. This time around, we assume that $x \not\in \Vect(e_1^\star,e_2^\star,e_3^\star)$.
By rescaling $x$ we can assume that the coefficient of $x$ on $e_4^\star$ equals $1$ in the basis $(e_1^\star,e_2^\star,e_3^\star,e_4^\star)$.
Now, we come back to matrices for greater convenience.
By considering the representing matrix of $x \otimes y$ in the basis $(e_1,\dots,e_n)$,
we deduce that $\calM$ contains a matrix $M$ with rank $1$, trace $0$, and of the shape
$$M=\begin{bmatrix}
A_0 & [0]_{4 \times (n-4)} \\
C_1 & [0]_{(n-4) \times (n-4)}
\end{bmatrix} \quad \text{with}
\quad
A_0=\begin{bmatrix}
\delta & ? & ? & 1 \\
0 & 0 & 0 & 0 \\
\beta & ? & ? & 1 \\
\gamma & ? & ? & ?
\end{bmatrix} \quad \text{for some $(\delta,\beta,\gamma) \in \F^3$,}$$
and $C_1 \in \Mat_{n-4,4}(\F)$.

Now, thanks to our loosened assumption on $\calM$ we can subtract from $M$ a matrix
of the form
$$\begin{bmatrix}
A_1 & [0]_{4 \times (n-4)} \\
C_1 & [0]_{(n-4) \times (n-4)}
\end{bmatrix} \quad \text{with}
\quad
A_1=\begin{bmatrix}
0 & 0 & 0 & 0 \\
0 & 0 & 0 & 0 \\
\beta & 0 & 0 & 0 \\
\gamma & ? & ? & 0
\end{bmatrix}$$
and deduce that $\calM$ contains a trace zero matrix of the form
$$M'=\begin{bmatrix}
\delta & [?]_{1 \times 3} & [0]_{1 \times (n-4)} \\
[0]_{3 \times 1} & B_0 & [0]_{3 \times (n-4)} \\
[0]_{(n-4) \times 1} & [0]_{(n-4) \times 3} & [0]_{(n-4) \times (n-4)}
\end{bmatrix},$$
where
$$B_0=\begin{bmatrix}
0 & 0 & 0 \\
? & ? & 1 \\
? & ? & ?
\end{bmatrix}.$$
By choosing $\lambda \in \F \setminus \{0,\delta\}$ and by adding to $M'$ the diagonal matrix $0_1 \oplus \lambda I_2 \oplus 0_{n-3}$,
we obtain a matrix in $\calM$ for which $\delta,\lambda,0$ are eigenvalues (this uses the assumption $n \geq 5$ in a critical way), which requires that $\delta=0$.
Since $A_0$ has rank $1$, it readily follows that $\beta=\gamma=0$, to the effect that no modification was actually necessary
at the start: hence the matrix $B_0$ is simply the lower-right block of $A_0$. We note in particular that $B_0$ has rank $1$ and trace zero.
Then we observe, thanks to the starting assumptions on $\calM$,
that the matrix $B_0+(N \oplus 0_1)$ has at most two eigenvalues in $\F$ for all $N \in \mathfrak{sl}_2(\F)$.
Hence Lemma \ref{lemma:lastblock} yields that the last row of $B_0$ is zero (obviously the last column is not zero).
In particular, since $A_0$ has rank $1$ and trace zero we now have the much simplified form
$$A_0=\begin{bmatrix}
0 &  \varepsilon & 0 & 1 \\
0 & 0 & 0 & 0 \\
0 & \varepsilon & 0 & 1 \\
0 & 0 & 0 & 0
\end{bmatrix} \quad \text{for some $\varepsilon \in \F$.}$$
Now, we deduce that there exists a nonzero scalar $\eta \in \F^\times$ (remembering the loosened assumption from the start of the proof) such that
for all $(\lambda,\mu,a) \in \F^3$, the matrix
$$K=\begin{bmatrix}
0 & \varepsilon & 0 & 1 \\
\eta a & 0 & \mu & 0 \\
0 & \lambda & 0 & 1 \\
0 & 0 & a & 0
\end{bmatrix}$$
has at most one nonzero eigenvalue in $\F$. Yet one computes that its characteristic polynomial is
$$t^4+ (a(\eta\varepsilon+1)+\lambda \mu) t^2+a^2 \eta (\varepsilon+\lambda).$$
In order to conclude, it will suffice to observe that the parameters can be chosen so that this polynomial has two distinct nonzero roots in $\F$.
To do so, we take arbitrary distinct nonzero elements $\gamma_1$ and $\gamma_2$ in $\F$ and we
write $(t^2-\gamma_1^2)(t^2-\gamma_2^2)=t^4+gt^2+h$ for some $(g,h)\in \F \times \F^\times$.
Then we choose $a \in \F^\times$ that is not a root of $\eta \varepsilon t^2 - h$
(this is possible because $|\F|>3$ and the polynomial under consideration is nonzero, as $h \neq 0$).
Then we adjust $\lambda \in \F$ so that $\eta a^2 (\varepsilon+\lambda)=h$, and the previous constraint on $a$ requires that
$\lambda \neq 0$. Finally, since $\lambda \neq 0$, we can then adjust $\mu$ so that $a(\eta\varepsilon+1)+\lambda \mu=g$,
thereby obtaining a contradiction.

\vskip 3mm
\noindent \textbf{Step 3: The final contradiction.} \\
Now we know that $x \in \Vect(e_1^\star,e_2^\star,e_3^\star)$. Hence
the matrix of $\calM$ that corresponds to $x \otimes y$ is of the form
$\begin{bmatrix}
R_0 & [0]_{3 \times (n-3)} \\
[?]_{(n-3) \times 3} & [0]_{(n-3) \times (n-3)}
\end{bmatrix}$, where the first row of $R_0 \in \Mat_3(\F)$ is nonzero (because $y \not\in \Vect(e_2,\dots,e_n)$).
By using the starting assumption on $\calM$ once more, and by extracting, we deduce that the linear combinations of
$R_0$ and of the matrices of the form
$$\begin{bmatrix}
0 & [0]_{1 \times 2} \\
C & N
\end{bmatrix}, \quad \text{with $N \in \mathfrak{sl}_2(\F)$ and $C \in \F^2$,}$$
are $1^\star$-spec.
Yet as the first row of $R_0$ is nonzero the space of all such matrices has dimension $6$: this contradicts Theorem \ref{theorem:maintheodim1starspec}
(which has already been proved in Section \ref{section:largefields} for $3$-by-$3$ matrices).

This final contradiction shows that all the non-$\calS$-adapted vectors belong to the union $H_1 \cup H_2$.
This completes the proof of the Third Confinement Lemma.

Note that we have proved a slightly more precise result: the non-$\calS$-adapted vectors
all belong to the union $\Vect(e_1,e_4,\dots,e_n) \cup \Vect(e_2,e_3,e_4,\dots,e_n)$, but it turns out we do not need such a precise result.

\end{document}